\documentclass[twoside,leqno]{article}

\usepackage[letterpaper]{geometry}

\usepackage{siamproceedings}

\usepackage[T1]{fontenc}
\usepackage{amssymb,amsfonts,color,xcolor}
\usepackage{graphicx,tikz}
\usetikzlibrary{arrows.meta,
	calc,
	positioning,
	quotes}
\usetikzlibrary{shapes.geometric}
\usepackage{epstopdf}
\usepackage{enumerate}
\usepackage{breakcites}
\usepackage{algorithmic}
\usepackage{amsopn,bbm}
\usepackage[round]{natbib}

\hypersetup{
  colorlinks,
  citecolor={blue!60!black},
  urlcolor={blue!70!black},
  linkcolor={red!60!black},
  breaklinks,
  hypertexnames=false
}
\usepackage[norefs,nocites]{refcheck}

\newsiamremark{remark}{Remark}
\newsiamremark{hypothesis}{Hypothesis}
\crefname{hypothesis}{Hypothesis}{Hypotheses}
\newsiamthm{claim}{Claim}

\newcommand{\ci}{\mbox{${}\perp\mkern-11mu\perp{}$}}
\newtheorem{innercustomasm}{Assumption}
\newenvironment{customasm}[1]
  {\renewcommand\theinnercustomasm{#1}\innercustomasm}
  {\endinnercustomasm}

\DeclareMathOperator*{\argmin}{argmin}
\DeclareMathOperator*{\argmax}{argmax}

\DeclareMathOperator{\supp}{supp}

\begin{document}

\title{\Large Nonparametric inference under shape constraints: past, present and future}
    \author{Richard J. Samworth\thanks{Statistical Laboratory, University of Cambridge (\email{r.samworth@statslab.cam.ac.uk}, \url{http://www.statslab.cam.ac.uk/~rjs57/})}}

\date{}

\maketitle

\begin{abstract}
We survey the field of nonparametric inference under shape constraints, providing a historical overview and a perspective on its current state.  An outlook and some open problems offer thoughts on future directions.
\end{abstract}

\section{Introduction.}

Traditionally, we think of statistical methods as being divided into parametric approaches, which can be restrictive, but where estimation is typically straightforward (e.g.~using maximum likelihood), and nonparametric methods, which are more flexible but often require careful choices of tuning parameters.  Nonparametric inference under shape constraints sits somewhere in the middle, seeking in some ways the best of both worlds.  The origins of the field are often traced  to \citet{grenander1956theory}, who proved that there exists a unique maximum likelihood estimator (MLE) of a decreasing density on the non-negative half-line (and was able to characterise it explicitly).  Thus, even though the class of decreasing densities is infinite-dimensional, statistical estimation can proceed in a familiar fashion, with no tuning parameters to choose.

Through the remainder of the 20th century and into the first 10-15 years of the current millennium, the field evolved in several different directions.  On the one hand, the monotonic constraint was incorporated into other core statistical problems, such as regression \citep{ayer1955empirical,brunk1955maximum,vaneeden1956maximum} and hazard function estimation \citep{rao1970estimation}.  Theoreticians were enticed by the non-standard cube-root rates of convergence and risk bounds \citep{rao1969estimation,groeneboom1985estimating,birge1987estimating,birge1989grenander,zhang2002risk,chatterjee2015risk}, while the Pool Adjacent Violators Algorithm provided a linear time algorithm for computation \citep{brunk1972statistical}.  Convex regression and density estimation then became the next natural challenge \citep{groeneboom2001estimation,guntuboyina2015global}, while S-shaped function estimation is a more recent topic \citep{feng2022nonparametric}.  Further developments and historical references are provided in the books by \citet{brunk1972statistical}, \citet{robertson1988order} and \citet{groeneboom2014nonparametric}, as well as the 2018 special issue of the journal {\emph{Statistical Science}} \citep{samworth2018special}. 

Over the last 10-15 years or so, problems in \emph{multivariate} shape-constrained inference have received significant focus.  In particular, the estimation of \emph{log-concave densities}, i.e.~those densities $f$ for which $\log f$ is concave, has emerged as a central topic within the field.  This definition works equally well in $d$ dimensions as in the univariate case, and moreover the class is closed under marginalisation, conditioning, convolution and linear transformations, making it a very natural infinite-dimensional generalisation of the class of Gaussian densities.  Once again, a unique MLE exists, so we retain the attraction of a fully automatic, nonparametric procedure.  On the other hand, since the characterisation of the MLE is now less explicit, considerable effort has been devoted to its efficient computation.  The period from roughly 2010 to the early 2020s saw rapid and exciting developments in our understanding of log-concave density estimation and related  problems such as multivariate isotonic regression \citep{han2019isotonic,deng2020isotonic,pananjady2022isotonic} and convex regression in $d \geq 2$ dimensions \citep{kur2024convex}.  %For instance, it is now known that shape-constrained estimators can often not only achieve optimal minimax rates of convergence, but also achieve faster (sometimes even parametric) rates of convergence over subclasses of densities with additional structure.  

Sections~\ref{Sec:Grenander} and~\ref{Sec:LogConcaveDensityEstimation} provide a brief tour of results in shape-constrained inference up to the last year or two, focusing on the Grenander estimator and log-concave density estimation.  However, now that most of the key questions related to the core topics of density estimation and regression have been answered, the field has moved in another interesting direction.  We have witnessed a significant broadening of the scope of shape-constrained ideas and techniques, so that they are now incorporated as part of more elaborate statistical tasks.  To illustrate these latest developments, we present an application of shape-constrained inference in linear regression due to \citet{feng2025optimal} in Section~\ref{Sec:LinearRegression}.  Here, the goal is to improve on the ordinary least squares estimator when the error density is non-Gaussian, via an $M$-estimator with a data-driven, convex loss function, designed to minimise the asymptotic variance of the resulting estimator of the vector of regression coefficients.  In Section~\ref{Sec:OtherExamples}, we briefly mention two other examples of very recent ways in which shape constraints have been assimilated into modern statistical methods, in subgroup selection \citep{muller2025isotonic} and conditional independence testing \citep{hore2025testing}.  We conclude with an outlook and some open problems.     

The following notation is used throughout the paper.  For $n \in \mathbb{N}$, we write $[n] := \{1,\ldots,n\}$, and for $x \in \mathbb{R}$, we write $x_+ := \max(x,0)$ and $x_- := \max(-x,0)$.  The Euclidean norm is denoted $\|\cdot\|$.  If  $(\mathcal{X},\mathcal{A})$ is a measurable space and $P,Q$ are probability measures on $\mathcal{X}$, then their \emph{total variation distance} is $\mathrm{TV}(P,Q) := \sup_{A \in \mathcal{A}} |P(A) - Q(A)|$.  If $P,Q$ have densities $p,q$ with respect to a $\sigma$-finite measure $\mu$ on $(\mathcal{X},\mathcal{A})$, then we define their \emph{Hellinger distance} by $\mathrm{H}(P,Q) := \bigl\{\int_{\mathcal{X}} (p^{1/2} - q^{1/2})^2 \, d\mu\bigr\}^{1/2}$ and the \emph{Kullback--Leibler divergence} from $Q$ to $P$ by $\mathrm{KL}(P,Q) := \int_{\mathcal{X}} p \log(p/q) \, d\mu$.

\section{The Grenander estimator}
\label{Sec:Grenander}

Let $\mathcal{G}$ denote the set of all left-continuous, decreasing densities $g:(0,\infty)\rightarrow [0,\infty)$.  This is an infinite-dimensional convex set under pointwise addition and scalar multiplication.  Our first goal is to introduce a modern approach to shape-constrained estimation via a population-level projection framework.  For a general probability measure $Q$ on $(0,\infty)$, let $\mathcal{G}_Q$ be the set of all $g \in \mathcal{G}$ for which the log-likelihood functional
\[
  L(g,Q) := \int_{(0,\infty)} \log g \,dQ
\] 
is well-defined, i.e.~at least one of $\int_{(0,\infty)}(\log g)_+\,dQ$ and $\int_{(0,\infty)}(\log g)_-\,dQ$ is finite, with $\log 0:=-\infty$.  %It is not always true that $\mathcal{G}_Q = \mathcal{G}$; see the first assertion of Proposition~\ref{thm:Lstar} below and the example following the statement of the theorem. 

In Proposition~\ref{thm:Lstar} below we characterise precisely those probability measures $Q$ on $(0,\infty)$ for which
\[
  L^*(Q):=\sup_{g\in\mathcal{G}_Q}L(g,Q)
\]
is finite.  This will allow us to establish in Theorem~\ref{thm:GrenProj} that for such $Q$, there exists a unique maximiser of $g \mapsto L(g,Q)$ over $\mathcal{G}$, namely the left derivative of the least concave majorant of the distribution function of~$Q$.  In the case where $Q$ is the empirical distribution of independent and identically distributed positive random variables, such a maximiser is the MLE.  %The proof of Theorem~\ref{thm:Lstar} makes crucial use of the fact that the function $x\mapsto 1/x$ on $(0,\infty)$ is the pointwise envelope function of the class $\mathcal{G}$ of decreasing densities.

\begin{proposition}[\citealp{samworth2025modern}]
\label{thm:Lstar}
Let $Q$ be a Borel probability measure on $(0,\infty)$. Then $\mathcal{G}_Q=\mathcal{G}$ if and only if $\int_{(0,\infty)}(\log x)_-\,dQ(x)<\infty$. Moreover, we have the following trichotomy:
\begin{enumerate}[\itshape(a)]
\item If $\int_{(0,\infty)}(\log x)_+\,dQ(x)=\infty$, then $L^*(Q)=-\infty$.
% not always the case that $\mathcal{G}_Q=\mathcal{G}$
\item If $\int_{(0,\infty)}(\log x)_+\,dQ(x)<\infty=\int_{(0,\infty)}(\log x)_-\,dQ(x)$,
% $\int_{(0,\infty)}\log(1/x)\,dQ(x)=\infty$,
% $\mathcal{G}_Q\neq\mathcal{G}$ in this case
then $L^*(Q)=\infty$.
\item If $\int_{(0,\infty)}|\log x|\,dQ(x)<\infty$, 
% $\mathcal{G}_Q=\mathcal{G}$ in this case
then $L^*(Q)\in\mathbb{R}$.
\end{enumerate}
\end{proposition}
%Note for example that $\int_{(0,\infty)}(\log x)_-\,dQ(x)=\infty$ when $Q$ has density $f$ on $(0,\infty)$ given by $f(x)= \frac{\log 2}{x\log^2 x} \mathbbm{1}_{\{x\in (0,1/2]\}}$, in which case $\int_{(0,\infty)}\log g \, dQ$ is not well-defined when $g \in \mathcal{G}$ is given by $g(x)=x^{-1/2}\mathbbm{1}_{\{x\in (0,1/4]\}}$ for instance, and $L^*(Q)=\infty$ by Proposition~\ref{thm:Lstar}\emph{(b)}. If instead $Q$ has density $f$ on $(0,\infty)$ given by $f(x) = \frac{\log 2}{x\log^2 x}\mathbbm{1}_{\{x\in (2,\infty)\}}$, then $\int_{(0,\infty)}\log x \, dQ(x)=\infty$, whence $\mathcal{G}_Q=\mathcal{G}$ and $L^*(Q)=-\infty$ by Proposition~\ref{thm:Lstar}\emph{(a)}.  
The proof of Proposition~\ref{thm:Lstar} is based on the fact that $\sup_{g \in \mathcal{G}} g(x) = 1/x$ for all $x \in (0,\infty)$.  

Given a distribution function $G$ on $(0,\infty)$, it is convenient to let $\mathrm{ldlcm}(G)$ denote the left derivative of its least concave majorant.  We are now in a position to state our main projection result.
\begin{theorem}[\citealp{samworth2025modern}]
\label{thm:GrenProj}
For a Borel probability measure $Q$ on $(0,\infty)$ with distribution function~$G$, let $g^*\equiv g^*(Q) := \mathrm{ldlcm}(G)$.  If $\int_{(0,\infty)}|\log x|\,dQ(x)<\infty$, then 
%\begin{equation}
%\label{eq:GrenProj1}
\[
g^* = \argmax_{g\in\mathcal{G}}L(g,Q).
\]
%\end{equation}
Moreover, $\sup\{x \in (0,\infty):g^*(x) > 0\} = \inf\{x\in (0,\infty):G(x)=1\}\in (0,\infty]$, and $g^*$ is constant on any interval $(a,b]$ with $Q\bigl((a,b)\bigr) = 0$.
%
%\begin{enumerate}[\itshape(a)]
%%\renewcommand{\theenumi}{(\alph{enumi})}
%\item $\mathrm{supp}(Q^*)=(0,x^*] \cap (0,\infty)$, where $x^*=\inf\{x\in (0,\infty):F(x)=1\}\in (0,\infty]$, and $g^*$ is constant on any interval $(a,b]$ with $Q\bigl((a,b)\bigr) = 0$.
%\item The differential entropy $-\int_0^\infty g^*\log g^*$ of the distribution $Q^*$ is well-defined in $[-\infty,\infty]$ if and only if $\int_{(0,\infty)}\log x \, dQ(x)$ is well-defined in $[-\infty,\infty]$, in which case
%\begin{equation}
%\label{eq:Lgstar}
%L^*(Q)=L(g^*,Q)=L(g^*,Q^*)=\int_0^\infty g^*\log g^*.
%\end{equation}
%\item Suppose that $\int_{(0,\infty)}|\log x|\,dQ(x)<\infty$, or equivalently that $L^*(Q)$ is finite. Then
%\begin{equation}
%\label{eq:GrenProj1}
%g^*\equiv g^*(Q)=\argmax_{g\in\mathcal{G}}L(g,Q).
%\end{equation}
%\item If in addition $Q$ has a Lebesgue density $g_0$ with $\int_0^\infty g_0\,|\log g_0|<\infty$ and $\int_0^\infty g_0(x)\,|\log x|\,dx < \infty$, then $\mathrm{KL}(Q,Q^*)<\infty$ and 
%\begin{equation}
%\label{eq:GrenProj2}
%Q^*=\argmin_{\tilde{Q}\in\mathcal{Q}}\mathrm{KL}(Q,\tilde{Q})%.
%\end{equation}
%\end{enumerate}
% $\inf\emptyset=\infty$ by convention
\end{theorem}
A consequence on Theorem~\ref{thm:GrenProj} is that if $Q$ has Lebesgue density $g_0$ satisfying $\int_0^\infty g_0\,|\log g_0|<\infty$ and $\int_0^\infty g_0(x)\,|\log x|\,dx < \infty$, then $\mathrm{KL}(g_0,g^*)<\infty$ and 
%\begin{equation}
%\label{eq:GrenProj2}
\[
g^*=\argmin_{g\in\mathcal{G}}\mathrm{KL}(g_0,g).
\]
%\end{equation}
This explains our `projection' terminology: under the above conditions, $g^*$ minimises the Kullback--Leibler divergence from $\mathcal{G}$ to $g_0$.  Of course, if $g_0 \in \mathcal{G}$, then $g^* = g_0$.
%Note that if $Q$ has a decreasing Lebesgue density $g_0 \in \mathcal{G}$, then for any density $g$ on $(0,\infty)$ such that $\int_0^\infty g_0\log g \in (-\infty,\infty]$, we have
%\[
%  \int_{(0,\infty)}\log g\,dQ \leq \int_0^\infty g_0\log g+\int_0^\infty g_0\log\Bigl(\frac{g_0}{g}\Bigr) = \int_{(0,\infty)}\log g_0\,dQ
%\]
%by the non-negativity of Kullback--Leibler divergence.  Moreover, the conclusion $\int_{(0,\infty)}\log g\,dQ \leq \int_{(0,\infty)}\log g_0\,dQ$ also holds trivially if $\int_0^\infty g_0\log g = -\infty$.  Equality holds if and only if $g=g_0$ almost everywhere, so if $g$ is also left continuous, then equality holds if and only if $g=g_0$. Thus,~\eqref{eq:Lgstar}--\eqref{eq:GrenProj2} hold in this special case with $Q^*=Q$ and $g^*=g_0$. More generally, when $Q$ has a Lebesgue density $g_0\notin\mathcal{G}$ with finite entropy,~\eqref{eq:GrenProj2} asserts that $Q^*$ is the unique closest distribution to $Q$ in a Kullback--Leibler sense among distributions with decreasing densities. 
We therefore refer to $Q \mapsto g^*(Q)$ as the \emph{Grenander projection}\index{Grenander projection}.

In the special case where $X_1,\ldots,X_n$ are independent, positive random variables with empirical distribution~$\mathbb{Q}_n$, the \emph{Grenander estimator} is $\hat{g}_n := g^*(\mathbb{Q}_n)$; see Figure~\ref{Fig:Grenander}.  The least concave majorant $\mathbb{G}_n^*$ of the empirical distribution function $\mathbb{G}_n$, and hence its left derivative, can be computed using the Pool Adjacent Violators Algorithm (PAVA), which requires only $O(n)$ computational time and storage.
\begin{figure}
  \centering
  \includegraphics[width=0.9\textwidth]{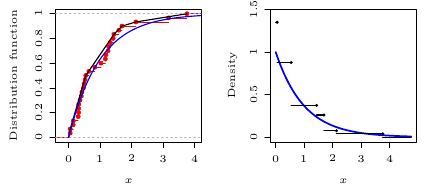}
  \caption{\label{Fig:Grenander}Left: The empirical distribution $\mathbb{G}_n$ (red) of a sample of size $n=30$ from the $\mathrm{Exp}(1)$ distribution, whose distribution function is the blue curve.  The solid black line is the least concave majorant $\mathbb{G}_n^*$ of $\mathbb{G}_n$.  Right: The $\mathrm{Exp}(1)$ density (blue) and the Grenander estimator (black).}
\end{figure}

\subsection{Theoretical properties of the Grenander estimator}

The following analytic result on least concave majorants shows that the Grenander projection $Q\mapsto Q^*$ is 1-Lipschitz with respect to the \emph{Kolmogorov distance}, which for probability measures $Q_1,Q_2$ is defined in terms of their distribution functions $G_1,G_2$ by
\[
  d_{\mathrm{K}}(Q_1,Q_2)=\sup_{x\in (0,\infty)}|G_1(x)-G_2(x)| =: \|G_1-G_2\|_\infty.
\]
\vspace{-0.2cm}
\begin{lemma}[\citealp{marshall1970discussion}]
%\label{Lemma:Marshall}
Let $G_1,G_2$ be two distribution functions on $(0,\infty)$ and let $G_1^*,G_2^*$ be their respective least concave majorants. Then
$\|G_1^*-G_2^*\|_\infty\leq \|G_1-G_2\|_\infty$.
\end{lemma}
\begin{proof}
Let $d:=\|G_1-G_2\|_\infty$. Then $G_1\leq G_2+d\leq G_2^*+d$ on $(0,\infty)$, and $G_2^*+d$
% or $G_2+d$
is concave, so $G_1^*\leq G_2^*+d$ on $(0,\infty)$ by the definition of $G_1^*$ as the least concave majorant of $G_1$. By symmetry, $G_2^*\leq G_1^*+d$, so $\|G_1^*-G_2^*\|_\infty\leq d$, as required.
\end{proof}
In combination with the Glivenko--Cantelli theorem, Marshall's lemma immediately yields the first part of Corollary~\ref{cor:Robust} below; the second part is a consequence of basic properties of the left derivatives of concave functions \citep[][Lemma~9.6]{samworth2025modern}.  
\begin{corollary}%[{\citealp[Lemma~5.5]{Pat01}}]
\label{cor:Robust}
Let $X_1,X_2,\ldots \stackrel{\mathrm{iid}}{\sim} Q$ on $(0,\infty)$ with distribution function $G$, and write $G^*$ for its least concave majorant, with corresponding distribution $Q^*$. For $n \in \mathbb{N}$, let $\mathbb{Q}_n$ denote the empirical distribution of $X_1,\ldots,X_n$ with corresponding empirical distribution function $\mathbb{G}_n$, and write $\mathbb{G}_n^*$ for its least concave majorant, with corresponding distribution $\mathbb{Q}_n^*$. Let $\hat{g}_n:=g^*(\mathbb{Q}_n)$ and $g^* := g^*(Q)$.  
\begin{enumerate}[(a)]
\item We have
\[
  d_{\mathrm{K}}(\mathbb{Q}_n^*,Q^*)=\|\mathbb{G}_n^*-G^*\|_\infty \leq \|\mathbb{G}_n - G\|_\infty \stackrel{\mathrm{a.s.}}{\rightarrow} 0
\]
as $n \rightarrow \infty$.
\item Moreover, $\hat{g}_n(x_0) \stackrel{\mathrm{a.s.}}{\rightarrow} g^*(x_0)$ for every continuity point $x_0 \in (0,\infty)$ of $g^*$, and $\mathrm{TV}(\hat{g}_n,g^*) \stackrel{\mathrm{a.s.}}{\rightarrow} 0$ as $n \rightarrow \infty$.  Finally, if $g^*$ is continuous on $(0,\infty)$ then for each $x_0 \in (0,\infty)$,
    \[
      \sup_{x \geq x_0} |\hat{g}_n(x) - g^*(x)| \stackrel{\mathrm{a.s.}}{\rightarrow} 0.
    \]
\end{enumerate}
%\end{enumerate}
\end{corollary}
% \begin{proof}
% \emph{(a)} This follows immediately from Marshall's lemma and the Glivenko--Cantelli theorem. 

% \emph{(b)} These claims follow immediately from the Glivenko--Cantelli theorem and Lemma~\ref{Lemma:ConcaveFunctions}.  
% \end{proof}
From Corollary~\ref{cor:Robust}, we see that the Grenander estimator is consistent under correct model specification (when $Q=Q^*$ has a density $g_0=g^*\in\mathcal{G}$) and robust under misspecification in the sense that it converges in the modes described to the Grenander projection of $Q$.  

To conclude this section, we present minimax risk bounds.  For $H,L > 0$, let $\mathcal{G}(H,L)$ denote the set of left-continuous, decreasing densities on $(0,L]$ that are bounded by $H$.  Note that the $\mathrm{Exp}(1)$ density for instance does not belong to $\mathcal{G}(H,L)$ for any $H,L$.  Let~$\tilde{\mathcal{G}}_n$ denote the set of estimators of $g_0$ based on $X_1,\ldots,X_n$, i.e.~the set of Borel measurable functions from $(0,\infty)^n$ to the set of integrable functions on $(0,\infty)$.  
\begin{theorem}[\citealp{birge1987estimating,birge1989grenander}]
\label{Thm:GrenanderRisk}
Let $X_1,\ldots,X_n \stackrel{\mathrm{iid}}{\sim} g_0 \in \mathcal{G}$.  Then, writing $S := \log(1 +HL)$, we have
  \[
0.0975S^{1/3} \leq \inf_{\tilde{g}_n \in \tilde{\mathcal{G}}_n} \sup_{g_0 \in \mathcal{G}(H,L)} n^{1/3} E_{g_0}\bigl\{\mathrm{TV}(\tilde{g}_n,g_0)\bigr\} \leq 0.975S^{1/3}
\]
when $S \geq 1.31$ and $n \geq 39S$.
\end{theorem}
The quantity $S$ that appears in Theorem~\ref{Thm:GrenanderRisk} is affine invariant, as is the total variation loss function.  We remark that the condition $S \geq 1.31$ is only used in the lower bound, and the Grenander estimator achieves the upper bound.  Thus, under these side conditions, the worst-case total variation risk over $\mathcal{G}(H,L)$ of the Grenander estimator comes within a factor of 10 of the best achievable risk.  %The cube-root rate of convergence is perhaps unsurprising in view of the fact that decreasing densities are differentiable Lebesgue almost everywhere in $(0,\infty)$, and   

\section{Log-concave density estimation}
\label{Sec:LogConcaveDensityEstimation}

Although the class of decreasing densities on $(0,\infty)$ provides a natural starting point for studying shape-constrained estimation problems, with an explicit expression for the maximum likelihood estimator, the family is nevertheless limited, and the ideas do not generalise particularly straightforwardly to multivariate settings.  The class of log-concave densities, on the other hand, contains many commonly-encountered parametric families, and has several closure and stability properties that make it a very natural infinite-dimensional generalisation of the class of Gaussian densities.  We will study the basic properties of this class in the next subsection, and will subsequently discuss questions of statistical estimation.  

\subsection{Definition and basic properties}

We say $f:\mathbb{R}^d \rightarrow [0,\infty)$ is \emph{log-concave}\index{log-concave function} if $\log f$ is concave, with the convention that $\log 0 := -\infty$.  Examples of univariate log-concave densities include Gaussian densities, Gumbel densities, logistic densities, $\Gamma(\alpha,\lambda)$ densities with $\alpha \geq 1$, $\mathrm{Beta}(a,b)$ densities with $a,b \geq 1$ and Laplace densities.  Multivariate Gaussian densities are also log-concave, as are uniform densities on convex, compact sets, densities with independent log-concave components and spherically symmetric densities of the form $x \mapsto g(\|x\|)$, where $g:[0,\infty) \rightarrow [0,\infty)$ is decreasing and log-concave.  Log-concave densities $f$ are unimodal, i.e.~the super-level set $\{x \in \mathbb{R}^d:f(x) \geq t\}$ is convex for every $t \in \mathbb{R}$, and have exponentially-decaying tails.  Thus, Cauchy densities are not log-concave, and it can be shown that the density of the Gaussian mixture $p N_d(\mu_1,I) + (1-p) N_d(\mu_2,I)$ is log-concave when $p \in (0,1)$ if and only if $\|\mu_1-\mu_2\| \leq 2$.  In contrast to the class $\mathcal{G}$ of Section~\ref{Sec:Grenander}, there is no requirement for any aspect of the support of the densities in the class to be known.  A helpful univariate characterisation is the following:
\begin{lemma}[\citealp{ibragimov1956composition}]
%\label{Lemma:Ibragimov}
A density $f$ on $\mathbb{R}$ is log-concave if and only if the convolution $f \ast g$ is unimodal for every unimodal density $g$.
\end{lemma}
It will be convenient to let $\mathcal{F}_d$ denote the class of upper semi-continuous, log-concave densities on $\mathbb{R}^d$.  The densities here are with respect to $d$-dimensional Lebesgue measure; the upper semi-continuity restriction fixes a particular version of the density (the set of discontinuities of a log-concave density lie on the boundary of a convex set, so have zero Lebesgue measure).  The Gaussian mixture example in the previous paragraph shows that $\mathcal{F}_d$ is not a convex set, unlike the class $\mathcal{G}$ of decreasing densities on $(0,\infty)$ studied in Section~\ref{Sec:Grenander}; fortunately, and perhaps surprisingly, this turns out to cause fewer difficulties for estimation than one might imagine.  Now let $\Phi$ denote the convex set of upper semi-continuous, concave functions $\phi:\mathbb{R}^d \rightarrow [-\infty,\infty)$ that are coercive in the sense that $\phi(x) \rightarrow -\infty$ as $\|x\| \rightarrow \infty$.  Since $\log f$ is coercive whenever $f \in \mathcal{F}_d$, we therefore have
\[
      \mathcal{F}_d = \biggl\{e^\phi:\phi \in \Phi,\int_{\mathbb{R}^d} e^\phi = 1\biggr\}.
    \]

Densities $f \in \mathcal{F}_d$ with a fixed scale necessarily satisfy certain pointwise bounds.  For such $f$, we let $\mu_f := \int_{\mathbb{R}^d} x f(x) \, dx$ and $\Sigma_f := \int_{\mathbb{R}^d} (x-\mu_f)(x-\mu_f)^\top f(x) \, dx$.  For $\mu \in \mathbb{R}^d$ and positive definite $\Sigma \in \mathbb{R}^{d \times d}$, we also write $\mathcal{F}_d^{\mu,\Sigma} := \{f \in \mathcal{F}_d: \mu_f=\mu,\Sigma_f=\Sigma\}$. 
\begin{lemma}
\label{Lemma:Envelope}
\begin{enumerate}[(a)]
  \item \textbf{Univariate case}: For every $f \in \mathcal{F}_1^{0,1}$, we have
\[
  f(x) \leq \left\{ \begin{array}{ll} \frac{1}{(2-x^2)^{1/2}} & \mbox{if $x \in [-1,1]$} \\
                                      e^{-|x|+1}   & \mbox{otherwise.} \end{array} \right.
                                \]
  \item \textbf{Multivariate case}: There exist $A_d > 0$, $B_d \in \mathbb{R}$, both depending only on $d$, such that for every $f \in \mathcal{F}_d^{0,I}$, we have
    \[
f(x) \leq e^{-A_d\|x\| + B_d}.
\]
Moreover, for every $x \in \mathbb{R}^d$ with $\|x\| \leq 1/9$, we have $f(x) \geq 2^{-8d}$.
                                  \end{enumerate}
\end{lemma}
Lemma~\ref{Lemma:Envelope}\emph{(a)} is due to \citet{feng2021adaptation}; the upper bound in~\emph{(b)} was proved by \citet{fresen2013multivariate} and the lower bound is due to \citet{lovasz2007geometry}.  The lemma provides an `envelope' function for the isotropic elements of the class $\mathcal{F}_d$ (i.e.~those with mean zero and identity covariance matrix).  When $d=1$ and $x \in [-1,1]$ the bound on the envelope function is sharp, and when $x \notin [-1,1]$, it is almost sharp, in the sense that $\sup_{f \in \mathcal{F}_1^{0,1}} f(x) \geq e^{-(|x|+1)}$ (since the densities of the $\mathrm{Exp}(1) - 1$ and $1 - \mathrm{Exp}(1)$ distributions both belong to~$\mathcal{F}_1^{0,1}$) and moreover $e^{|x|-1}\sup_{f \in \mathcal{F}_1^{0,1}} f(x) \rightarrow 1$ as $|x| \rightarrow \infty$.  The fact that the envelope functions in Lemma~\ref{Lemma:Envelope} are integrable turns out to be very convenient in studying the rates of statistical estimation over $\mathcal{F}_d$ (see Section~\ref{Sec:Holder} below); in particular, we will not need to make further restrictions on the class to state minimax risk bounds.  Again, this is in contrast to the class $\mathcal{G}$ studied in Section~\ref{Sec:Grenander}, which has the non-integrable envelope function $x \mapsto 1/x$ on $(0,\infty)$, and for which we introduced the subclass $\mathcal{G}(H,L)$ in Theorem~\ref{Thm:GrenanderRisk}.

Having understood something about the shape of log-concave densities, we now turn to their stability properties:
\begin{theorem}[\citealp{prekopa1973logarithmic,prekopa1980logarithmic}]
\label{Thm:Prekopa}
  Let $d = d_1 + d_2$ and let $f:\mathbb{R}^d \rightarrow [0,\infty)$ be log-concave.  Then
  \[
    x \mapsto \int_{\mathbb{R}^{d_2}} f(x,y) \, dy
  \]
  is log-concave on $\mathbb{R}^{d_1}$.
\end{theorem}
Theorem~\ref{Thm:Prekopa} immediately tells us that marginals of log-concave densities are log-concave.  As another application, we have the following:
\begin{corollary}
 % \label{Cor:Conv}
If $f,g$ are log-concave densities on $\mathbb{R}^d$, then $f \ast g$ is a log-concave density on $\mathbb{R}^d$.
\end{corollary}
\begin{proof}
The function $(x,y) \mapsto f(x-y)g(y)$ is log-concave on $\mathbb{R}^{2d}$, so the result follows from Theorem~\ref{Thm:Prekopa}.
\end{proof}
The proof of our final stability result is a straightforward exercise.
\begin{proposition}
  Let $X$ have a log-concave density on $\mathbb{R}^d$.
  \begin{enumerate}[\itshape(a)]
  \item If $A \in \mathbb{R}^{m \times d}$, with $m \leq d$ and $\mathrm{rank}(A) = m$, then $AX$ has a log-concave density on $\mathbb{R}^m$.
  \item If $X = (X_1^\top,X_2^\top)^\top$, then the conditional density of $X_1$ given $X_2=x_2$ is log-concave for every $x_2$.
\end{enumerate}
    \end{proposition}
    Thus, as mentioned in the introduction, the class of log-concave densities is closed under linear transformations, marginalisation, conditioning and convolution, just as is the class of Gaussian densities.  On the other hand, there are senses in which $\mathcal{F}_d$ is much larger than the Gaussian class.  For instance, for a bivariate Gaussian random vector $(X,Y)$, we know that the regression function $x \mapsto \mathbb{E}(Y|X=x)$ is a necessarily an affine function.  But now let $h_1:[0,1] \rightarrow (-\infty,0]$ be convex with $h_1(0) = h_1(1) = 0$ and $h_2:[0,1] \rightarrow [0,\infty)$ be concave with $h_2(0) = h_2(1) = 0$.  Suppose further that it is not the case that both $h_1$ and $h_2$ are identically zero.  Then the uniform density on the set $\{(x,y) \in [0,1] \rightarrow \mathbb{R}: h_1(x) \leq y \leq h_2(x)\}$ belongs to $\mathcal{F}_2$, but if $(X,Y)$ has this density, then the regression function is $x \mapsto \{h_1(x)+h_2(x)\}/2$.  In particular, the class of possible regression functions includes the set of functions that are differences of convex functions on $[0,1]$, which is the same as the set of functions having left and right derivatives that are of bounded variation on every compact sub-interval of $(0,1)$.
    
    \subsection{Log-concave projections}
 %   \label{SubSec:LogConcProj}

As we saw for the Grenander estimator, the key to understanding statistical questions related to log-concave density estimation lies in a notion of projection.  
    For $\phi \in \Phi$ and an arbitrary Borel probability measure $P$ on $\mathbb{R}^d$, define the log-likelihood-type functional
    \begin{equation}
      \label{Eq:LogLikelihoodFunctional}
      L(\phi,P) := \int_{\mathbb{R}^d} \phi \, dP - \int_{\mathbb{R}^d} e^\phi + 1,
    \end{equation}
    and let $L^*(P) := \sup_{\phi \in \Phi} L(\phi,P)$.   The second and third terms in~\eqref{Eq:LogLikelihoodFunctional} account for the fact that the elements of $\Phi$ are not constrained to be log-densities, though it is interesting to note that a Lagrange multiplier is not required here.  Indeed, if $L(\phi,P) \in \mathbb{R}$ and $\int_{\mathbb{R}^d} e^\phi > 0$, and if we define $\phi + c$ pointwise for $c \in \mathbb{R}$, then
    \[
      \frac{\partial}{\partial c}L(\phi+c,P) = 1 - e^c \int_{\mathbb{R}^d} e^\phi,
    \]
    so $L(\phi+c,P)$ is maximal when $c = -\log\bigl(\int_{\mathbb{R}^d} e^\phi\bigr)$.  In other words, if $\phi^* \in \argmax_{\phi \in \Phi} L(\phi,P)$ with $L^*(P) \in \mathbb{R}$ and $\int_{\mathbb{R}^d} e^{\phi^*} > 0$, then $\phi^*$ is a log-density.  On the other hand, if $\int_{\mathbb{R}^d} e^\phi = 0$ and $L(\phi,P) > -\infty$, then by choosing~$c$ arbitrarily large we see that $L^*(P) = \infty$.

    Theorem~\ref{Thm:LCProjection} below provides the crucial characterisation of the existence and uniqueness of log-concave projection.  Let $\mathcal{P}_d$ denote the set of probability measures on $\mathbb{R}^d$ for which $\int_{\mathbb{R}^d} \|x\| \, dP(x) < \infty$ and $P(H) < 1$ for all hyperplanes $H \subseteq \mathbb{R}^d$.  The \emph{convex support}\index{convex support} of a Borel probability measure $P$ on $\mathbb{R}^d$, denoted $\mathrm{csupp}(P)$, is the intersection of all closed, convex sets $C \subseteq \mathbb{R}^d$ with $P(C) = 1$.  The \emph{interior}\index{interior of a set} of a set $S \subseteq \mathbb{R}^d$, denoted $\mathrm{int} \, S$, is the union of all open sets contained in $S$.  Finally, the \emph{effective domain}\index{effective domain of a convex or concave function} of a concave function $\phi:\mathbb{R}^d \rightarrow [-\infty,\infty)$ is $\mathrm{dom}(\phi) := \{x \in \mathbb{R}^d:\phi(x) > -\infty\}$.  
\begin{theorem}[\citealp{cule2010maximum,cule2010theoretical,dumbgen2011approximation}]
\label{Thm:LCProjection}
Let $P$ be a Borel probability measure on $\mathbb{R}^d$.
\begin{enumerate}[\itshape(a)]
\item If $\int_{\mathbb{R}^d} \|x\| \, dP(x) = \infty$, then $L^*(P) = -\infty$.
\item If $\int_{\mathbb{R}^d} \|x\| \, dP(x) < \infty$ but $P(H)  = 1$ for some hyperplane $H$, then $L^*(P) = \infty$.
\item If $P \in \mathcal{P}_d$, then $L^*(P) \in \mathbb{R}$, and there exists a well-defined projection $\psi^*:\mathcal{P}_d \rightarrow \mathcal{F}_d$, given by
  \[
    \psi^*(P) := \argmax_{f \in \mathcal{F}_d} \int_{\mathbb{R}^d} \log f \, dP.
  \]
  Moreover, $\mathrm{int} \, \mathrm{csupp}(P) \subseteq \mathrm{dom}\bigl(\log \psi^*(P)\bigr) \subseteq \mathrm{csupp}(P)$.
\end{enumerate}
\end{theorem}

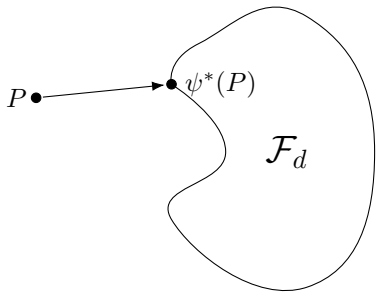
\begin{figure}[htbp!]
\begin{center}
  \begin{tikzpicture}[scale = 0.9]
    \filldraw[black] (-4,0.8) circle (2pt) node[anchor=east]{$P$};
    \filldraw[black] (-2,1) circle (2pt) node[anchor=west]{\,$\psi^*(P)$};
    \node[text width=16pt] at (-0.3,0) {\Large $\mathcal{F}_d$};
    \draw [-Latex](-3.9,0.81) -- (-2.1,0.99);
  \draw  plot[smooth, tension=0.9] coordinates {(-2,1) (-1.2,0) (-2,-1) (0,-2) (1,0) (0,2) (-1.6,1.6) (-2,1)};
\end{tikzpicture}
\caption{\label{Fig:LCProjection}Illustration of the log-concave projection $\psi^*(P)$, which is well-defined when $P \in \mathcal{P}_d$, despite the non-convexity of $\mathcal{F}_d$.}
\end{center}
\end{figure}

It is part~\emph{(c)} of Theorem~\ref{Thm:LCProjection} that is particularly interesting: even though $\mathcal{F}_d$ is not a convex set, there is still a notion of log-concave projection\index{log-concave projection} via maximum likelihood; see Figure~\ref{Fig:LCProjection}.  In particular, the result provides a very natural way to fit log-concave densities to data.  Given $X_1,\ldots,X_n$ in $\mathbb{R}^d$ having $d$-dimensional convex hull~$C_n$ and empirical distribution $\mathbb{P}_n$, we can use the MLE $\hat{f}_n := \psi^*(\mathbb{P}_n)$.  The last part of Theorem~\ref{Thm:LCProjection} reveals that this density estimator is supported on $C_n$.

One of the great attractions of the log-concave maximum likelihood estimator is that, in contrast to kernel density estimation methods or other traditional nonparametric smoothing techniques, there are no tuning parameters to choose.  On the other hand, in general there is no closed-form expression for $\hat{f}_n$, so optimisation algorithms are required to compute the estimator.

When $d=1$, an Active Set algorithm can be used to compute the log-concave maximum likelihood estimator very efficiently \citep{dumbgen2007active,dumbgen2011logcondens}; up to machine precision, it terminates with the exact solution in finitely many steps.  On the other hand, when $d \geq 2$, the feasible set is much more complicated, and only slower algorithms are available.  For $y = (y_1,\ldots,y_n)^\top \in \mathbb{R}^n$, let $\bar{h}_y:\mathbb{R}^d \rightarrow [-\infty,\infty)$ denote the smallest concave function with $\bar{h}_y(X_i) \geq y_i$ for $i \in [n]$; these are called \emph{tent functions}\index{tent function} in the literature (see Figure~\ref{Fig:Schematic}).  It can be shown that the log-concave MLE belongs to the class of tent functions, which is finite-dimensional.  We can therefore write the objective function in terms of the tent pole heights $y = (y_1,\ldots,y_n)^\top$~as 
\[
\tau(y) \equiv \tau(y_1,\ldots,y_n) := \frac{1}{n}\sum_{i=1}^n \bar{h}_y(X_i) - \int_{C_n} \exp\{\bar{h}_y(x)\} \, dx.
\]

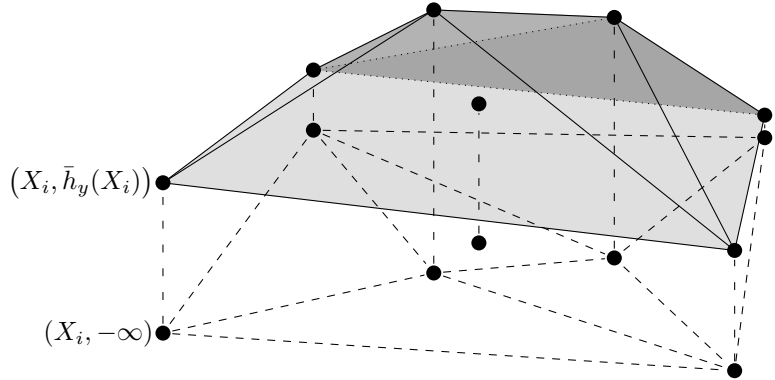
\begin{figure}
\centering
  \begin{tikzpicture}
 \node (v1) at (0,-2) {};
    \node (v2) at (0,0) {};
    \node (v3) at (2,0.7) {};
    \node (v4) at (2,1.5) {};
    \node (v5) at (3.6,-1.2) {};
    \node (v6) at (3.6,2.3) {};
    \node (v7) at (4.2,-0.8) {};
    \node (v8) at (4.2,1.05) {};
    \node (v9) at (6,-1) {};
    \node (v10) at (6,2.2) {};
    \node (v11) at (7.6,-2.5) {};
    \node (v12) at (7.6,-0.9) {};
    \node (v13) at (8,0.6) {};
    \node (v14) at (8,0.9) {};
    \filldraw[color=gray!25](0,0) -- (3.6,2.3) -- (7.6,-0.9) -- (0,0);
    \filldraw[color=gray!25](0,0) -- (2,1.5) -- (3.6,2.3) -- (0,0);
    \filldraw[color=gray!25](3.6,2.3) -- (4.2,1.05) -- (7.6,-0.9);
    \filldraw[color=gray!25](3.6,2.3) -- (6,2.2) -- (7.6,-0.9);
    \filldraw[color=gray!25](6,2.2) -- (7.6,-0.9) -- (8,0.9);
    \filldraw[color=gray!60](2,1.5) -- (3.6,2.3) -- (6,2.2);
    \filldraw[color=gray!70](2,1.5) -- (6,2.2) -- (8,0.9);
    \fill (v1) circle (0.1) node [left] {$(X_i,-\infty)$};
    \fill (v2) circle (0.1) node [left] {$\bigl(X_i,\bar{h}_y(X_i)\bigr)$};
    \fill (v3) circle (0.1);
    \fill (v4) circle (0.1);
    \fill (v5) circle (0.1);
    \fill (v6) circle (0.1);
    \fill (v7) circle (0.1);
    \fill (v8) circle (0.1);
    \fill (v9) circle (0.1);
    \fill (v10) circle (0.1);
    \fill (v11) circle (0.1);
    \fill (v12) circle (0.1);
    \fill (v13) circle (0.1);
    \fill (v14) circle (0.1);
    %\fill[fill = red!30](v2) -- (v6) -- (v12) -- cycle;
    \draw[loosely dashed](v1) -- (v2);
    \draw[loosely dashed](v3) -- (v4);
    \draw[loosely dashed](v5) -- (v6);
    \draw[loosely dashed](v7) -- (v8);
    \draw[loosely dashed](v9) -- (v10);
    \draw[loosely dashed](v11) -- (v12);
    \draw[loosely dashed](v13) -- (v14);
    \draw[dashed](v1) -- (v11);
    \draw[dashed](v1) -- (v3);
    \draw[dashed](v1) -- (v5);
    \draw[dashed](v3) -- (v5);
    \draw[dashed](v3) -- (v9);
    \draw[dashed](v3) -- (v13);
    \draw[dashed](v5) -- (v9);
    \draw[dashed](v5) -- (v11);
    \draw[dashed](v9) -- (v11);
    \draw[dashed](v9) -- (v13);
    \draw[dashed](v11) -- (v13);
    \draw[solid](0,0) -- (2,1.5);
    \draw[solid](0,0) -- (3.6,2.3);
    \draw[solid](0,0) -- (7.6,-0.9);
    \draw[solid](2,1.5) -- (3.6,2.3);
    \draw[solid](3.6,2.3) -- (6,2.2);
    \draw[solid](3.6,2.3) -- (7.6,-0.9);
    \draw[solid](6,2.2) -- (7.6,-0.9);
    \draw[solid](6,2.2) -- (8,0.9);
    \draw[solid](7.6,-0.9) -- (8,0.9);
    \draw[dotted](2,1.5) -- (6,2.2);
    \draw[dotted](2,1.5) -- (8,0.9);
    %    \filldraw[black] (0,0) circle (2.5pt) node[anchor=south]{$(1-p)/4$};
%    \filldraw[black] (4,0) circle (2.5pt) node[anchor=south]{$(1-p)/4$};
%    \filldraw[black] (2,-2) circle (2.5pt) node[anchor=north]{$p$};
%    \filldraw[black] (0,-4) circle (2.5pt) node[anchor=north]{$(1-p)/4$};
%    \filldraw[black] (4,-4) circle (2.5pt) node[anchor=north]{$(1-p)/4$};
%    \draw (0,0) -- (4,0) -- (4,-4) -- (0,-4) -- (0,0);
 \end{tikzpicture}
 \caption{\label{Fig:Schematic}A schematic picture of a tent function in the case $d=2$.}
\end{figure}

\noindent This function is hard to optimise over $(y_1,\ldots,y_n)^\top \in \mathbb{R}^n$.  A key observation, however, is that we can define the modified objective function
%\begin{equation}
%\label{Eq:sigma}
\[
\sigma(y) \equiv \sigma(y_1,\ldots,y_n) := \frac{1}{n}\sum_{i=1}^n y_i - \int_{C_n} \exp\{\bar{h}_y(x)\} \, dx.
\]
%\end{equation}
Thus $\sigma \leq \tau$, but the crucial points are that $\sigma$ is concave and its unique maximum $\hat{y} \in \mathbb{R}^n$ satisfies $\log \hat{f}_n = \bar{h}_{\hat{y}}$.  Even though $\sigma$ is non-differentiable, a subgradient of $-\sigma$ can be computed at every point.  This motivates the use of Shor's $r$-algorithm, as well as methods based on Nesterov and randomised smoothing, to compute~$\hat{f}_n$ \citep{cule2009logconcdead,chen2024new}.  See Figures~\ref{Fig:Schematic} and~\ref{Fig:2D}.  %\citet{KoenkerMizera2010} study an alternative approximate approach based on imposing concavity of the discrete Hessian matrix of the log-density on a grid, and using a Riemann approximation to the integrability constraint.

\begin{figure}
\centering
\includegraphics[width=0.45\textwidth]{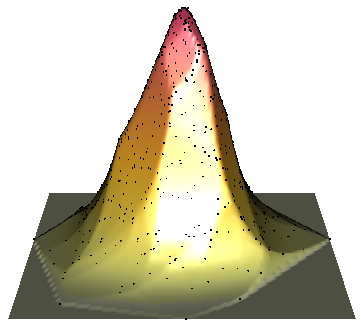} \hspace{0.2cm}
\includegraphics[width=0.45\textwidth]{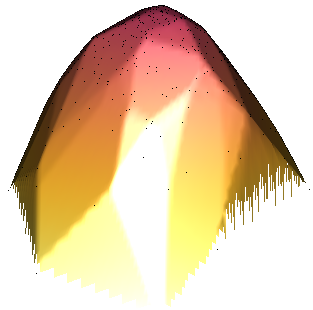}
\caption{\label{Fig:2D}The log-concave maximum likelihood estimator (left) and its logarithm (right) based on 1000 observations from a standard bivariate normal distribution.}
\end{figure}

\subsection{Properties of log-concave projections}

Although, for a general distribution $P \in \mathcal{P}_d$, the log-concave projection $\psi^*(P)$ does not have a closed form, we can nevertheless say quite a lot about its properties, starting with affine equivariance:

\begin{lemma}[\citealp{dumbgen2011approximation}]
\label{Lemma:Affine}
Let $X \sim P \in \mathcal{P}_d$, let $A \in \mathbb{R}^{d \times d}$ be invertible, let $b \in \mathbb{R}^d$, and let $P_{A,b}$ denote the distribution of $AX + b$.  Then $P_{A,b} \in \mathcal{P}_d$ and
\[
\psi^*(P_{A,b})(x) = \frac{1}{|\det A|} \psi^*(P)\bigl(A^{-1}(x-b)\bigr).
\]
\end{lemma}
Lemma~\ref{Lemma:Affine} tells us that log-concave projection commutes with invertible affine transformations $T$: writing $P^*$ for the distribution corresponding to $\psi^*(P)$, and with $X \sim P$ and $X^* \sim P^*$, we have $T(X)^* \stackrel{d}{=} T(X^*)$. 

We would like the log-concave projection to preserve as many properties of the original distribution as possible.  Indeed, such preservation results have motivated several associated methodological developments, including the \emph{smoothed log-concave MLE} \citep{dumbgen2009maximum,chen2013smoothed} and a new approach to independent component analysis \citep{samworth2012independent}.  One result in this direction can be obtained from the first-order stationarity conditions.
\begin{lemma}[\citealp{dumbgen2011approximation}]
\label{Lemma:Basic}
Let $P \in \mathcal{P}_d$, let $\phi^* := \log \psi^*(P)$, and let $P^*(B) := \int_B e^{\phi^*}$ for any Borel set $B \subseteq \mathbb{R}^d$.  If $\Delta:\mathbb{R}^d \rightarrow [-\infty,\infty)$ is such that $\phi^* + t\Delta \in \Phi$ for some $t > 0$, then 
\[
\int_{\mathbb{R}^d} \Delta \, dP \leq \int_{\mathbb{R}^d} \Delta \, dP^*.
\]
\end{lemma}
% \begin{proof}
%   We have
%   \[
%     0 \geq \limsup_{t \searrow 0} \frac{L(\phi^* + t\Delta,P) - L(\phi^*,P)}{t} = \int_{\mathbb{R}^d} \Delta \, dP - \liminf_{t \searrow 0} \int_{\mathbb{R}^d} \frac{e^{\phi^*+t\Delta} - e^{\phi^*}}{t}.
%   \]
%   Let $t_0 > 0$ be such that $\phi^* + t_0\Delta \in \Phi$ and initially assume that $\int_{\mathbb{R}^d} \Delta \, dP^* > -\infty$.  Then for $t \in (0,t_0]$,
%   \[
% \biggl|\frac{e^{\phi^*+t\Delta} - e^{\phi^*}}{t}\biggr| \leq \frac{e^{\phi^*+t_0\Delta} - e^{\phi^*}}{t_0}\mathbbm{1}_{\{\Delta \geq 0\}} + (-\Delta)e^{\phi^*}\mathbbm{1}_{\{\Delta < 0\}},
%   \]
% so the result follows by dominated convergence.  On the other hand, if $\int_{\mathbb{R}^d} \Delta \, dP^* = -\infty$, then for $t \in (0,t_0]$,
% \[
% \int_{\mathbb{R}^d} \frac{e^{\phi^*+t\Delta} - e^{\phi^*}}{t} \leq \int_{\mathbb{R}^d} \frac{e^{\phi^*+t_0\Delta} - e^{\phi^*}}{t_0}\mathbbm{1}_{\{\Delta \geq 0\}} + \int_{\mathbb{R}^d} \Delta\mathbbm{1}_{\{\Delta < 0\}} \, dP^* = -\infty,
% \]
% so $\int_{\mathbb{R}^d} \Delta \, dP = -\infty$, as required.
% \end{proof}
As a special case of Lemma~\ref{Lemma:Basic}, we obtain 
\begin{corollary}
\label{Cor:Moment}
Let $P \in \mathcal{P}_d$.  Then $P$ and the log-concave projection measure~$P^*$ from Lemma~\ref{Lemma:Basic} are convex ordered in the sense that 
\[
\int_{\mathbb{R}^d} h \, dP^* \leq \int_{\mathbb{R}^d} h \, dP
\]
for all convex $h:\mathbb{R}^d \rightarrow (-\infty,\infty]$.  
\end{corollary}
Applying Corollary~\ref{Cor:Moment} to $h(x) = t^\top x$ for arbitrary $t \in \mathbb{R}^d$ allows us to conclude that $\int_{\mathbb{R}^d} x \, dP^*(x) = \int_{\mathbb{R}^d} x \, dP(x)$; in other words, log-concave projection preserves the mean $\mu$ of a distribution $P \in \mathcal{P}_d$.  On the other hand, we see that the projection shrinks the second moment, in the sense that $A := \int_{\mathbb{R}^d} (x-\mu)(x-\mu)^\top d(P-P^*)(x)$ is non-negative definite.  In fact, we can say more: from the convex ordering in Corollary~\ref{Cor:Moment} and Strassen's theorem \citep{strassen1965existence}, there exist random vectors $X \sim P$ and $X^* \sim P^*$, defined on the same probability space, such that $\mathbb{E}(X|X^*) = X^*$ almost surely.  Thus $\mathbb{E}\{X^*(X-X^*)^\top\} = 0$, so
\[
\mathrm{Cov}(X) = \mathrm{Cov}(X^* + X - X^*) = \mathrm{Cov}(X^*) + \mathrm{Cov}(X - X^*)
\] 
and we deduce that $A=0$ if \emph{and only if} $P$ has a log-concave density.  The smoothed log-concave MLE exploits Corollary~\ref{Cor:Moment} by defining the new estimator $\tilde{f}_n := \hat{f}_n \ast N_d(0,\hat{A})$, where $\hat{A}$ is the sample version of $A$ above.  This estimator remains log-concave, is smooth (real analytic), and matches the first two moments of the data.

\subsection{H\"older continuity and risk bounds}
\label{Sec:Holder}

As for the Grenander projection, the log-concave projection has a continuity property, but this is a little more involved and we will require a little preparatory work.  Given Borel probability distributions $P, Q$ on $\mathbb{R}^d$, their \emph{Wasserstein distance}\index{Wasserstein distance} is defined as
\[
  d_{\mathrm{W}}(P,Q) := \inf_{(X,Y) \sim (P,Q)} \mathbb{E}\|X-Y\|,
\]
where the infimum is taken over all pairs $(X,Y)$, defined on the same probability space, with $X \sim P$ and $Y \sim Q$.  
Further, whenever $P$ has mean $\mu \in \mathbb{R}^d$ and $X \sim P$, we define 
\[
  \epsilon_P := \inf_{u \in \mathbb{S}_{d-1}} \mathbb{E}\bigl\{|u^\top(X-\mu)|\bigr\},
\]
where $\mathbb{S}_{d-1} := \{u \in \mathbb{R}^d:\|u\|=1\}$ denotes the Euclidean unit sphere.  The quantity $\epsilon_P$ can be thought of as a robust analogue of the minimum eigenvalue of the covariance matrix of the distribution $P$ (note that its definition does not require $P$ to have a finite second moment).  We can also interpret $\epsilon_P$ as measuring the extent to which~$P$ avoids placing all its mass on a single hyperplane.  

\begin{theorem}[\citealp{barber2021local}]
  \label{Thm:Holder}
  Whenever $P \in \mathcal{P}_d$, we have $\epsilon_P > 0$.  Moreover, there exists $C_d^\circ > 0$, depending only on $d$, such that for all $P,Q \in \mathcal{P}_d$, we have
%  \begin{equation}
%\label{Eq:HolderContinuity}
\[
\mathrm{H}\bigl(\psi^*(P),\psi^*(Q)\bigr) \leq C_d^\circ\biggl\{\frac{d_{\mathrm{W}}(P,Q)}{\max(\epsilon_P,\epsilon_Q)}\biggr\}^{1/4}.
\]
%  \end{equation}
\end{theorem}
Theorem~\ref{Thm:Holder} states that the log-concave projection is locally H\"older-$(1/4)$ continuous, when considered as a metric space map from $(\mathcal{P}_d,d_{\mathrm{W}})$ to $(\mathcal{F}_d,\mathrm{H})$.  It is natural to ask whether the map is in fact globally H\"older continuous, but in fact it is not even uniformly continuous: for instance, let $P_n = U[-1/n,1/n]$ and $Q_n = U[-1/n^2,1/n^2]$.  Then $d_{\mathrm{W}}(P_n,Q_n) = \frac{1}{2n} - \frac{1}{2n^2} \rightarrow 0$, but since $\psi^*(P_n) = \frac{n}{2}\mathbbm{1}_{\{[-1/n,1/n]\}}$ and $\psi^*(Q_n) = \frac{n^2}{2}\mathbbm{1}_{\{x \in [-1/n^2,1/n^2]\}}$, we have
\[
\mathrm{H}\bigl(\psi^*(P_n),\psi^*(Q_n)\bigr) = 2 - \frac{2}{n^{1/2}} \nrightarrow 0
\]
as $n \rightarrow \infty$.  In this counterexample, we have $\max(\epsilon_{P_n},\epsilon_{Q_n}) = 1/(2n) \rightarrow 0$, so there is no contradiction of Theorem~\ref{Thm:Holder}.  Moreover, it turns out that the exponent $1/4$ in Theorem~\ref{Thm:Holder} cannot be improved in general.  
% Another interesting aspect of Theorem~\ref{Thm:Holder} comes from considering the effect of affine transformations of a distribution $P \in \mathcal{P}_d$: given an invertible $A \in \mathbb{R}^d$, let $P_A$ denote the distribution of $AX$ when $X \sim P$.  By Lemma~\ref{Lemma:Affine} and the fact that Hellinger distance is (like all $f$-divergences) invariant to affine transformations, we can improve the bound in~\eqref{Eq:HolderContinuity} to
% \[
% \mathrm{H}\bigl(\psi^*(P),\psi^*(Q)\bigr) \leq C_d^\circ \cdot \inf_{A \in \mathbb{R}^d:\mathrm{rank}(A) = d} \biggl\{\frac{d_{\mathrm{W}}(P_A,Q_A)}{\max(\epsilon_{P_A},\epsilon_{Q_A})}\biggr\}^{1/4}.
% \]

One of the main attractions of the quantitative nature of the continuity result in Theorem~\ref{Thm:Holder} is that it facilitates a very general risk bound for the log-concave MLE as an estimator of the log-concave projection of the underlying distribution that holds even in misspecified settings.  We state the result in the univariate case for simplicity, and for $q \geq 1$ and $P \in \mathcal{P}_1$, write $\mu_q(P) := \bigl\{\int_{-\infty}^\infty |x|^q \, dP(x)\bigr\}^{1/q}$.
\begin{theorem}[\citealp{barber2021local}]
\label{Thm:MisspecifiedRate}
    Let $n \geq 2$, and let $X_1,\ldots,X_n \stackrel{\mathrm{iid}}{\sim} P \in \mathcal{P}_1$, with empirical distribution $\mathbb{P}_n$.  Fix $q > 1$. 
    \begin{enumerate}[\itshape(a)]
    \item \textbf{Upper bound}:  Suppose that $\mu_q(P) \leq M_q$.  Then there exists $C_q' > 0$, depending only on $q$, such that
    \[
 \mathbb{E}\bigl\{\mathrm{H}^2\bigl(\psi^*(\mathbb{P}_n),\psi^*(P)\bigr)\bigr\} \leq C_q' \cdot \sqrt{\frac{M_q}{\epsilon_P}} \cdot \frac{1 + (\log n)\mathbbm{1}_{\{q=2\}}}{n^{\frac{q-1}{2q}}}.
\]
\item \textbf{Lower bound}: There exist universal constants $\epsilon_*, c > 0$ such that
  \[
    \sup_{P \in \mathcal{P}_1: \mu_q(P) \leq 1, \epsilon_P \geq \epsilon_*} \mathbb{E}\bigl\{\mathrm{H}^2\bigl(\psi^*(\mathbb{P}_n),\psi^*(P)\bigr)\bigr\} \geq c \cdot \frac{1}{n^{\frac{q-1}{2q}}}.
  \]
  \end{enumerate}
\end{theorem}
The lower bound in Theorem~\ref{Thm:MisspecifiedRate} confirms that the dependence on $n$ in the upper bound on the worst-case performance of the log-concave MLE is sharp, except possibly for the additional logarithmic factor in the special case $q=2$.  Nevertheless, it is natural to ask whether this rate can be improved in the correctly specified case.  We write $\tilde{\mathcal{F}}_n$ for the set of all Borel measurable functions from $(\mathbb{R}^d)^{\times n}$ to the set of integrable functions on $\mathbb{R}^d$. 
\begin{theorem}[\citealp{kim2016global,kur2019optimality}]
  \label{Thm:FdRate}
  Let $d \in \mathbb{N}$, $n \geq d+1$, and let $X_1,\ldots,X_n \stackrel{\mathrm{iid}}{\sim} f_0 \in \mathcal{F}_d$, with empirical distribution $\mathbb{P}_n$.
  \begin{enumerate}[\itshape(a)]
\item \textbf{Upper bound}: Writing $\hat{f}_n := \psi^*(\mathbb{P}_n)$, there exists $C_d^* > 0$, depending only on $d$, such that
  \[
   \sup_{f_0 \in \mathcal{F}_d} E_{f_0}\bigl\{\mathrm{H}^2(\hat{f}_n,f_0)\bigr\} \leq C_d^* \cdot \left\{ \begin{array}{ll} n^{-4/5} & \mbox{if $d=1$} \\
n^{-2/(d+1)} \log n & \mbox{if $d \geq 2$.} \end{array} \right.
  \]
\item \textbf{Minimax lower bound}: There exists a universal constant $c_* > 0$ such that
\[
\inf_{\tilde{f}_n \in \tilde{\mathcal{F}}_n} \sup_{f_0 \in \mathcal{F}_d} E_{f_0}\bigl\{\mathrm{H}^2(\tilde{f}_n,f_0)\bigr\} \geq c_* \cdot \left\{ \begin{array}{ll} n^{-4/5} & \mbox{if $d=1$} \\ 
n^{-2/(d+1)} & \mbox{if $d \geq 2$.} \end{array} \right.
\]
 \end{enumerate}  
\end{theorem}
Theorem~\ref{Thm:FdRate} reveals that, in the case of correct model specification, the log-concave MLE attains the minimax optimal rate of convergence in squared Hellinger distance, up to the logarithmic factor when $d \geq 2$.  Nevertheless, the curse of dimensionality effect that is apparent in this result, combined with the computational challenges in higher dimensions, mean that one should regard the log-concave MLE as a low-dimensional estimator; see \citet{samworth2012independent}, \citet{xu2021high} and \citet{kubal2025log} for extensions to higher dimensions.  The phase transition at $d=2$ is surprising: since log-concave densities are twice differentiable Lebesgue almost everywhere, it had been expected that the rate would be the usual rate for estimating densities of smoothness $\beta=2$, namely $n^{-4/(d+4)}$.  However, log-concave densities can be badly behaved (discontinuous) on the boundary of their support, and it turns out that it is the difficulty of estimating this support that drives the rate in higher dimensions; in particular, the same minimax lower bound of order $n^{-2/(d+1)}$ holds when $d \geq 2$ if we restrict the supremum to the class of uniform densities on convex, compact sets. 

\subsection{Adaptation}
%\label{SubSec:Adaptation}

Although Theorems~\ref{Thm:MisspecifiedRate} and~\ref{Thm:FdRate} provide strong guarantees on the worst-case performance of the log-concave MLE, they ignore one of the appealing features of the estimator, namely its potential to adapt to certain characteristics of the unknown true density.  Here is one such result in the case $d=1$.  For $\beta \in (1,2]$ and $L > 0$, the H\"older class $\mathcal{H}(\beta,L)$ on an interval $I$ is the set of differentiable functions $\phi:I \rightarrow \mathbb{R}$ with
\[
|\phi'(x) - \phi'(x')| \leq L|x-x'|
\]
for $x,x' \in I$.  We also write $\phi \in \mathcal{H}(1,L)$ on $I$ if $\phi$ is $L$-Lipschitz on $I$.

%Recall that given an interval $I$, $\beta \in [1,2]$ and $L > 0$, we say $h:\mathbb{R} \rightarrow \mathbb{R}$ belongs to the H\"older class $\mathcal{H}_{\beta,L}(I)$ if for all $x,y \in I$, we have
%\begin{alignat*}{2}
% |h(x) - h(y)| &\leq L|x-y|,  &&\text{if} \ \beta=1 \\
%|h'(x) - h'(y)| &\leq L|x-y|^{\beta-1}, \quad &&\text{if} \ \beta > 1.
%\end{alignat*}
\begin{theorem}[\citealp{dumbgen2009maximum}]
%\label{Thm:Adapt0}
Let $X_1,\ldots,X_n \stackrel{\mathrm{iid}}{\sim} f_0 \in \mathcal{F}_1$, and assume that $\phi_0 := \log f_0 \in \mathcal{H}(\beta,L)$ on $I$ for some $\beta \in [1,2]$, $L > 0$ and compact interval $I \subseteq \mathrm{int} \, \mathrm{dom} (\phi_0)$.  Then
\[
\sup_{x_0 \in I} |\hat{f}_n(x_0) - f_0(x_0)| = O_p\biggl(\Bigl(\frac{\log n}{n}\Bigr)^{\beta/(2\beta+1)}\biggr).
\]
\end{theorem}
Here the log-concave MLE is adapting to unknown smoothness, in the sense that the upper bound on the rate improves with greater smoothness, even though the definition of the MLE does not depend on the unknown $\beta$.  %When measuring loss in the supremum norm, the need to restrict attention to a compact interval in the interior of support of $f_0$ is suggested by the right-hand plot in Figure~\ref{Fig:1D}.    
Other adaptation results are motivated by the thought that since the log-density of the MLE is piecewise affine, we might hope for faster rates of convergence in cases where $\log f_0$ is made up of a relatively small number of affine pieces.  We now describe two such results.  The first is based on a log-concave Marshall's lemma:
\begin{lemma}[\citealp{kim2018adaptation}]
\label{Lemma:LogConcaveMarshall}
Let $n \geq 2$, let $X_1,\ldots,X_n$ be real numbers that are not all equal, with empirical distribution function $\mathbb{F}_n$, let $\hat{F}_n(x) := \int_{-\infty}^x \hat{f}_n(t) \, dt$ for $x \in \mathbb{R}$, and let $F_0$ denote any distribution function whose corresponding density is concave on its support.  Then
\[
\|\hat{F}_n - F_0\|_\infty \leq 2\|\mathbb{F}_n - F_0\|_\infty.
\]
\end{lemma}
Now let $\mathcal{F}_{\mathrm{unif}}$ denote the class of uniform densities on a closed interval.
\begin{theorem}[\citealp{kim2018adaptation}]
\label{Thm:Adap1}
Let $n \geq 2$.  We have
\[
\sup_{f_0 \in \mathcal{F}_{\mathrm{unif}}} E_{f_0} \mathrm{TV}(\hat{f}_n,f_0) \leq \frac{4}{n^{1/2}}.
\]
\end{theorem}
\begin{proof}
The form of $f_0$ means that $\{x:\hat{f}_n(x) \geq f_0(x)\} = \{x:\log \hat{f}_n(x) \geq \log f_0(x)\}$ is an interval.  Hence
\begin{align*}
\mathrm{TV}(\hat{f}_n,f_0) &= \int_{x:\hat{f}_n(x) \geq f_0(x)\}} \bigl\{\hat{f}_n(x) - f_0(x)\bigr\} \, dx = \sup_{s \leq t} \int_s^t \bigl\{\hat{f}_n(x) - f_0(x)\bigr\} \, dx \nonumber \\
&= \sup_{s \leq t} \bigl\{\hat{F}_n(t) - \hat{F}_n(s) - F_0(t) + F_0(s)\bigr\} \leq 2\|\hat{F}_n - F_0\|_\infty \leq 4\|\mathbb{F}_n - F_0\|_\infty,
\end{align*}
by Lemma~\ref{Lemma:LogConcaveMarshall}.  It follows by the Dvoretsky--Kiefer--Wolfowitz--Massart--Reeve inequality \citep{dvoretzky1956asymptotic,massart1990tight,reeve2024short} that with $s^* := \bigl(\frac{8 \log 2}{n}\bigr)^{1/2}$,
\begin{align*}
E_{f_0} \mathrm{TV}(\hat{f}_n,f_0) &= \int_0^1 P_{f_0}\bigl(\mathrm{TV}(\hat{f}_n,f_0) \geq s\bigr) \, ds \leq s^* + 2\int_{s^*}^\infty e^{-ns^2/8} \, ds \leq \frac{4}{n^{1/2}},
\end{align*}
as required.
\end{proof}
\sloppy Using a more general version of Marshall's lemma than that stated in Lemma~\ref{Lemma:LogConcaveMarshall}, and observing that a concave function can cross a linear function at most twice, Theorem~\ref{Thm:Adap1} can be extended to cases where $f_0$ is log-linear on its support.  The leading constant $4$ in the previous bound deteriorates, however, as the slope of the log-linear density increases, and may need to be replaced with $6 \log n$ in the worst case.  %Going further, one can allow the true density $f_0$ to be arbitrary, and include an additional approximation error term that measures the proximity of $f_0$ to the class of log-concave densities that are log-linear on their support.

Generalising these ideas, for $k \in \mathbb{N}$ we define $\mathcal{F}^k$\index{log-$k$-affine log-concave densities} to be the class of log-concave densities $f$ on $\mathbb{R}$ for which $\log f$ is $k$-affine in the sense that there exist intervals $I_1,\ldots,I_k$ such that $f$ is supported on $I_1 \cup \cdots \cup I_k$, and $\log f$ is affine on each~$I_j$.  It will be convenient to define an empirical Kullback--Leibler loss\index{empirical Kullback--Leibler loss} for the log-concave MLE by
\[
  \widehat{\mathrm{KL}}(\hat{f}_n,f_0) := \frac{1}{n}\sum_{i=1}^n \log \frac{\hat{f}_n(X_i)}{f_0(X_i)}.
\]
Note here that the log-ratio of the estimator and the true density is averaged with respect to the empirical distribution, instead of with respect to $\hat{f}_n$.  For general densities $f$ and $g$, this does not make much sense as a loss function, because it would not be guaranteed to be non-negative.  However, an application of Lemma~\ref{Lemma:Basic} to the function $\Delta = \log (f_0/\hat{f}_n)$ yields that
%\begin{equation}
%  \label{Eq:KLhat}
\[
\mathrm{KL}(\hat{f}_n,f_0) \leq \widehat{\mathrm{KL}}(\hat{f}_n,f_0).
\]
%\end{equation}
In particular, an upper bound on $\widehat{\mathrm{KL}}(\hat{f}_n,f_0)$ immediately provides corresponding bounds on $\mathrm{TV}^2(\hat{f}_n,f_0)$, $\mathrm{H}^2(\hat{f}_n,f_0)$ and $\mathrm{KL}(\hat{f}_n,f_0)$.
\begin{theorem}[\citealp{kim2018adaptation}]
  \label{Thm:Adapt2}
Let $X_1,\ldots,X_n \stackrel{\mathrm{iid}}{\sim} f_0 \in \mathcal{F}_1$.  There exists a universal constant $C > 0$ such that for $n \geq 2$ and $f_0 \in \mathcal{F}_1$,
\[
E_{f_0}\bigl\{\widehat{\mathrm{KL}}(\hat{f}_n,f_0)\bigr\} \leq \min_{k \in [n]} \biggl\{\frac{Ck}{n}\log^{5/4} \Bigl(\frac{en}{k}\Bigr) + \inf_{f_k \in \mathcal{F}^k} \mathrm{KL}(f_0,f_k)\biggr\}.
\]
\end{theorem}
To help understand this theorem, first consider the case where $f_0 \in \mathcal{F}^k$.  Then $E_{f_0}\bigl\{\widehat{\mathrm{KL}}(\hat{f}_n,f_0)\bigr\} \leq \frac{Ck}{n}\log^{5/4} (en/k)$, which is nearly the parametric rate when $k$ is small.  More generally, this rate holds when $f_0 \in \mathcal{F}_1$ is only close to $\mathcal{F}^k$ in the sense that the approximation error $\mathrm{KL}(f_0,f_k)$ is $O\bigl(\frac{k}{n}\log^{5/4} \frac{en}{k}\bigr)$.  The result is known as a `sharp' oracle inequality\index{oracle inequality}\index{sharp oracle inequality}, because the leading constant for this approximation error term is 1.  It is worth noting that the techniques of proof, which rely on empirical process theory and local bracketing entropy bounds, are completely different from those used in the proof of Theorem~\ref{Thm:Adap1}.  It is also possible to state multivariate versions of Theorem~\ref{Thm:Adapt2} \citep{feng2021adaptation}, but the results are more complicated, and in particular depend not only on the number of log-affine pieces in the approximating log-concave density, but also on the sum of the number of facets in the polyhedral subdivision of its support into the regions on which it is log-affine.

\section{Linear regression via optimal convex \texorpdfstring{$M$}{M}-estimation}
\label{Sec:LinearRegression}

This section combines ideas from Sections~\ref{Sec:Grenander} and~\ref{Sec:LogConcaveDensityEstimation} in an eminently practical context.  In linear models, the Gauss--Markov theorem is the primary justification for the use of ordinary least squares (OLS) in settings where the Gaussianity of our error distribution may be in doubt.  It states that, provided the errors have a finite second moment, OLS attains the minimal covariance among all linear unbiased estimators.  On the other hand, it is now understood that biased, non-linear estimators can achieve lower mean squared error than OLS~\citep{stein1956inadmissibility}, especially when the noise distribution is appreciably non-Gaussian~\citep{zou2008composite}.

Consider a linear model where $Y_i = X_i^\top\beta_0 + \varepsilon_i$ for $i \in [n]$. Recall that an \textit{$M$-estimator} of $\beta_0 \in \mathbb{R}^d$ based on a loss function $\ell \colon \mathbb{R} \to \mathbb{R}$ is defined as an empirical risk minimiser
\begin{equation}
\label{eq:linreg-M-est}
\hat{\beta} \in \argmin_{\beta \in \mathbb{R}^d} \frac{1}{n}\sum_{i=1}^n \ell(Y_i - X_i^\top \beta), 
\end{equation}
provided that this exists. If $\ell$ is differentiable on $\mathbb{R}$ with negative derivative $\psi = -\ell'$, then $\hat{\beta} \equiv \hat{\beta}_\psi$ solves the corresponding estimating equations 
\begin{equation}
\label{eq:linreg-Z-est}
\frac{1}{n}\sum_{i=1}^n X_i\psi(Y_i - X_i^\top\hat{\beta}_\psi) = 0
\end{equation}
and is referred to as a \textit{$Z$-estimator}. We study a random design setting in which the pairs $(X_1,Y_1),\dotsc,(X_n,Y_n)$ are independent and identically distributed, with $X_1,\dotsc,X_n$ being $\mathbb{R}^d$-valued covariates that are independent of real-valued errors $\varepsilon_1,\ldots,\varepsilon_n$ having density $f_0$. Suppose further that $\mathbb{E}\{X_1\psi(\varepsilon_1)\} = 0$. This means that~$\hat{\beta}_\psi$ is \emph{Fisher consistent} in the sense that the population analogue of~\eqref{eq:linreg-Z-est} is satisfied by the true parameter $\beta_0$, i.e.~$\mathbb{E}\{X_1\psi(Y_1 - X_1^\top\beta_0)\} = 0$.  Then under suitable regularity conditions, including $\psi$ being differentiable and $\mathbb{E}(X_1 X_1^\top) \in \mathbb{R}^{d \times d}$ being invertible, we have
\begin{equation}
\label{eq:M-estimator-asymp}
\sqrt{n}(\hat{\beta}_\psi - \beta_0) \stackrel{d}{\rightarrow} N_d\bigl(0, V_{f_0}(\psi) \cdot \{\mathbb{E}(X_1 X_1^\top)\}^{-1}\bigr) \quad \text{as }n \to \infty, \;\text{where }V_{f_0}(\psi) := \frac{\mathbb{E}\psi^2(\varepsilon_1)}{\{\mathbb{E}\psi'(\varepsilon_1)\}^2}
\end{equation}
\citep[e.g.][]{vdV1998asymptotic}.

If the errors $\varepsilon_1,\varepsilon_2,\dotsc$ have a known absolutely continuous density $f_0$ on $\mathbb{R}$, then we can define the maximum likelihood estimator $\hat{\beta}^{\mathrm{MLE}}$ by taking $\ell = -\log f_0$ in~\eqref{eq:linreg-M-est}. In this case, $\psi = -\ell'$ is the \emph{score function (for location)}\footnote{The score is usually defined as a function of a parameter $\theta \in \mathbb{R}$ as the derivative of the log-likelihood; the link with our terminology comes from considering the location model $\{f_0(\cdot + \theta):\theta \in \mathbb{R}\}$, and evaluating the score at the origin.} $\psi_0 := (f_0'/f_0)\mathbbm{1}_{\{f_0 > 0\}}$.  Already at this stage, it will be helpful to observe that $\psi_0$ is decreasing if and only if $f_0$ is log-concave.  Under appropriate regularity conditions~\citep[e.g.][]{vdV1998asymptotic}, including that the \emph{Fisher information (for location)} $i(f_0) := \int_\mathbb{R} \psi_0^2\,f_0 = \int_{\{f_0 > 0\}}(f_0')^2/f_0$ is finite, we have
%\begin{equation}
%\label{eq:betahat-MLE-asymp}
\[
\sqrt{n}\,(\hat{\beta}^{\mathrm{MLE}} - \beta_0) \stackrel{d}{\rightarrow} N_d\biggl(0, \frac{\{\mathbb{E}(X_1 X_1^\top)\}^{-1}}{i(f_0)}\biggr)
\]
%\end{equation}
as $n \to \infty$. The limiting covariance matrix $\{\mathbb{E}(X_1 X_1^\top)\}^{-1}/i(f_0)$ constitutes the usual efficiency lower bound~\citep[Chapter~8]{vdV1998asymptotic}. Thus, $1/i(f_0)$ is the smallest possible value of the \textit{asymptotic variance factor} $V_{f_0}(\psi)$ in the limiting covariance of $\sqrt{n}(\hat{\beta}_\psi - \beta_0)$ in~\eqref{eq:M-estimator-asymp}.

\citet{feng2025optimal} seek to choose $\psi$ in a data-driven manner, such that the corresponding loss function~$\ell$ in~\eqref{eq:linreg-M-est} is convex, and such that the scale factor $V_{f_0}(\psi)$ in the asymptotic covariance~\eqref{eq:M-estimator-asymp} of the downstream estimator of $\beta_0$ is minimised.  Convexity is a particularly convenient property for a loss function, since for the purpose of $M$-estimation, it leads to more tractable theory and computation. 

Let $P_0$ be a probability distribution on $\mathbb{R}$ with a uniformly continuous density $f_0$.  
% A density $f_0$ is uniformly continuous on $\mathbb{R}$ if and only if it is continuous on $\mathbb{R}$ and $f_0(\pm\infty) = 0$.
Letting $\supp f_0 := \{z \in \mathbb{R} : f_0(z) > 0\}$, define $\mathcal{S}_0 \equiv \mathcal{S}(f_0) := \bigl(\inf(\supp f_0), \sup(\supp f_0)\bigr)$, which is the smallest open interval that contains $\supp f_0$. We write $\Psi_\downarrow(f_0)$ for the set of all $\psi \in L^2(P_0)$ that are decreasing and right-continuous, and observe that $\Psi_\downarrow(f_0)$ is a convex cone.  For $\psi \in \Psi_\downarrow(f_0)$ with $\int_\mathbb{R} \psi^2\,dP_0 > 0$, let
\begin{equation}
\label{eq:Vp0}
V_{f_0}(\psi) := \frac{\int_\mathbb{R}\psi^2\,dP_0}{\bigl(\int_{\mathcal{S}_0}f_0\,d\psi\bigr)^2} \in [0,\infty],
\end{equation}
where we have modified the denominator in~\eqref{eq:M-estimator-asymp} to extend the original definition to non-differentiable functions in $\Psi_\downarrow(f_0)$ such as $z \mapsto -\mathrm{sgn}(z)$.  As a first step towards minimising $V_{f_0}(\psi)$ over $\psi \in \Psi_\downarrow(f_0)$, note that $V_{f_0}(c\psi) = V_{f_0}(\psi)$ for every $c > 0$, so any minimiser is at best unique up to a positive scalar.  Ignoring unimportant edge cases where the denominator in~\eqref{eq:Vp0} is zero or infinity, our optimisation problem can therefore be formulated as a constrained minimisation of the numerator in~\eqref{eq:Vp0} subject to the denominator being equal to 1. This motivates the definition of
%\begin{equation}
%\label{eq:score-matching-obj}
\[
D_{f_0}(\psi) := \int_\mathbb{R} \psi^2\,dP_0 + 2\int_{\mathcal{S}_0} f_0\,d\psi \in [-\infty,\infty)
\]
%\end{equation}
for $\psi \in \Psi_\downarrow(f_0)$. If $\psi$ is locally absolutely continuous on $\mathcal{S}_0$ with derivative $\psi'$ Lebesgue almost everywhere, then
%\begin{equation}
%\label{eq:Dp0-deriv}
\[
D_{f_0}(\psi) = \int_\mathbb{R} \psi^2\,dP_0 + 2\int_{\mathcal{S}_0} \psi'f_0 = \int_\mathbb{R} (\psi^2 + 2\psi')\,dP_0 = \mathbb{E}\bigl(\psi^2(\varepsilon_1) + 2\psi'(\varepsilon_1)\bigr)
\]
%\end{equation}
when $\varepsilon_1 \sim P_0$; we recognise this as the \emph{score matching objective}~\citep{hyvarinen05score,song2021train}.

The formal link between $V_{f_0}(\cdot)$ and $D_{f_0}(\cdot)$ is that for $\psi \in \Psi_\downarrow(f_0)$ with $\int_\mathbb{R} \psi^2\,dP_0 > 0$, we have $\int_{\mathcal{S}_0}f_0\,d\psi \leq 0$ and $c\psi \in \Psi_\downarrow(f_0)$ for all $c \geq 0$, so
%\begin{equation}
%\label{eq:V-D-equivalence}
\[
\inf_{c \geq 0}D_{f_0}(c\psi) = \inf_{c \geq 0}\,\Bigl(c^2\int_\mathbb{R}\psi^2\,dP_0 + 2c\int_{\mathcal{S}_0}f_0\,d\psi\Bigr) = -\frac{\bigl(\int_{\mathcal{S}_0}f_0\,d\psi\bigr)^2}{\int_\mathbb{R}\psi^2\,dP_0} = -\frac{1}{V_{f_0}(\psi)}.
\]
%\end{equation}
Thus, minimising $V_{f_0}(\cdot)$ over $\Psi_\downarrow(f_0)$ is equivalent to minimising $D_{f_0}(\cdot)$ up to a scalar multiple, but $D_{f_0}(\cdot)$ is a convex function that is more tractable than $V_{f_0}(\cdot)$.  Further, when $f_0$ is absolutely continuous with $i(f_0) < \infty$,
\[
D_{f_0}(\psi) = \int_{\mathbb{R}} (\psi - \psi_0)^2\,dP_0 - \int_{\mathbb{R}} \psi_0^2\,dP_0 = \|\psi - \psi_0\|_{L^2(P_0)}^2 - i(f_0)
\]
for all $\psi \in \Psi_\downarrow(f_0)$, so 
%\begin{equation}
%\label{eq:Dp0-L2P0}
\[
\psi_0^* \in \argmin_{\psi \in \Psi_\downarrow(f_0)}D_{f_0}(\psi) = \argmin_{\psi \in \Psi_\downarrow(f_0)}\|\psi - \psi_0\|_{L^2(P_0)}^2.
\]
%\end{equation}
Thus, $\psi_0^*$ is a version of the \textit{$L^2(P_0)$-antitonic projection} of~$\psi_0$ onto $\Psi_\downarrow(f_0)$.

By exploiting this connection with score matching together with ideas from monotone function estimation similar to those in Section~\ref{Sec:Grenander}, \citet{feng2025optimal} prove that the solution to our asymptotic variance minimisation problem is the function~$\psi_0^*$ constructed explicitly in the following lemma.
\begin{lemma}
\label{lem:psi0-star}
Let $P_0$ be a distribution with a uniformly continuous density $f_0$ on $\mathbb{R}$. Let $F_0 \colon [-\infty,\infty] \to [0,1]$ be the corresponding distribution function, and for $u \in [0,1]$, define
\[
F_0^{-1}(u) := \inf\{z \in [-\infty,\infty] : F_0(z) \geq u\} \quad\text{and}\quad J_0(u) := (f_0 \circ F_0^{-1})(u).
\] 
% $Q^0$ is right-continuous with $\lim_{v \nearrow u}Q^_0(v) = F_0^{-1}(u)$, so $J_0$ is right-continuous with left limits given by $J_0(u-) := \lim_{v \nearrow u}J_0(v) = (f_0 \circ F_0^{-1})(u)$ for $u \in (0,1]$.
Writing $\hat{J}_0$ for the least concave majorant of $J_0$ on $[0,1]$, and $\hat{J}_0^{(\mathrm{R})}$ for the right derivative of $\hat{J}_0$, we have that 
\[
\psi_0^* := \hat{J}_0^{(\mathrm{R})} \circ F_0
\]
is decreasing and right-continuous as a function from $\mathbb{R}$ to $[-\infty,\infty]$, provided that we set $\hat{J}_0^{(\mathrm{R})}(1) := \lim_{u \nearrow 1}\hat{J}_0^{(\mathrm{R})}(u)$.
% $= \hat{J}_0^{(\mathrm{L})}(1)$
Moreover, $\psi_0^*(z) \in \mathbb{R}$ if and only if $z \in \mathcal{S}_0$.
\end{lemma}
We call $J_0$ the \emph{density quantile function} \citep{jones1992estimating}. When $f_0$ is the standard Cauchy density, Figure~\ref{fig:psi0-star-cauchy} illustrates $J_0$, and its least concave majorant $\hat{J}_0$, as well as the corresponding score functions $\psi_0 = f_0'/f_0$ and $\psi_0^*$.

\begin{figure}[htb]
\centering
\includegraphics[width=0.485\textwidth,trim={5.3cm 10cm 5.3cm 10cm},clip]{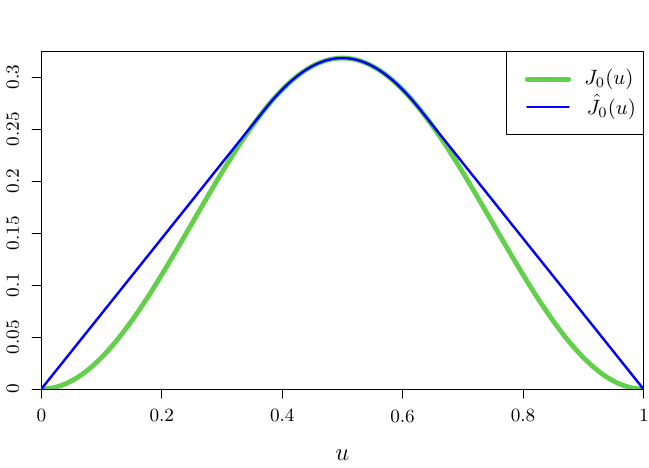}
\hfill
\includegraphics[width=0.485\textwidth,trim={5cm 9.82cm 5cm 10cm},clip]{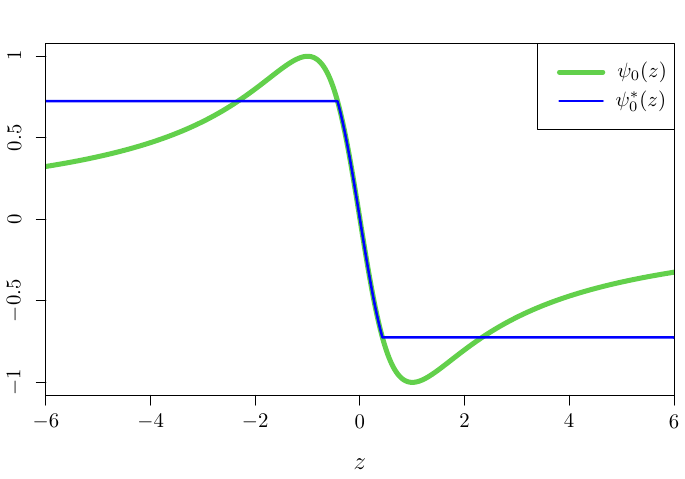}
\vspace{-0.3cm}
\caption{\emph{Left}: The density quantile function $J_0$ and its least concave majorant $\hat{J}_0$ for a standard Cauchy density. \emph{Right}: The corresponding score functions $\psi_0$ and $\psi_0^*$.}
\label{fig:psi0-star-cauchy}
\end{figure}

\begin{theorem}
%\label{thm:antitonic-score-proj}
In the setting of Lemma~\ref{lem:psi0-star}, the following statements hold.
\begin{enumerate}[(a)]
\item $\int_\mathbb{R}\psi_0^*\,dP_0 = 0$.
\item Suppose that $i^*(f_0) := \int_\mathbb{R}(\psi_0^*)^2\,dP_0 < \infty$.  Then $\psi_0^*$ is the unique minimiser of $D_{f_0}(\cdot)$ over $\Psi_\downarrow(f_0)$. Moreover, for every $\psi \in \Psi_\downarrow(f_0)$ satisfying $\int_\mathbb{R}\psi^2\,dP_0 > 0$, we have 
%\begin{equation}
%\label{eq:antitonic-fisher-LB}
\[
V_{f_0}(\psi) \geq V_{f_0}(\psi_0^*) = \frac{1}{i^*(f_0)} \in (0,\infty),
\]
%\end{equation}
with equality if and only if $\psi = \lambda\psi_0^*$ for some $\lambda > 0$.
\item If $f_0$ is absolutely continuous on $\mathbb{R}$, then $0 < i^*(f_0) \leq i(f_0)$, with equality if and only if $f_0$ is log-concave.  
\end{enumerate}
\end{theorem}
When $i^*(f_0) < \infty$, the \emph{antitonic relative efficiency} 
\[
\mathrm{ARE}^*(f_0) := \frac{i^*(f_0)}{i(f_0)}
\]
therefore quantifies the price we pay in statistical efficiency for insisting that our loss function be convex; this terminology is justified below.  It turns out that when $f_0$ is the Cauchy density, we have $\mathrm{ARE}^*(f_0) \geq 0.87$, showing that even though this heavy-tailed density is far from log-concave, the efficiency loss is surprisingly mild.

The antitonic score projection $\psi_0 \mapsto \psi_0^*$ yields a notion of projection of the corresponding density onto the log-concave class.  Importantly, however, this projection is different from the Kullback--Leibler (maximum likelihood) projection studied in Section~\ref{Sec:LogConcaveDensityEstimation}.  More precisely, when $f_0$ and $f_1$ are densities that are locally absolutely continuous on $\mathbb{R}$, the \emph{Fisher divergence} from $f_1$ to $f_0$ is defined as
\[
I(f_0,f_1) := 
\begin{cases}
\displaystyle \int_{\{f_0 > 0\}} \biggl(\Bigl(\log\frac{f_0}{f_1}\Bigr)'\biggr)^2\,f_0 \;&\text{if }\supp f_0 \subseteq \supp f_1 \\
\infty \;&\text{otherwise.}
\end{cases}
\]
The following lemma establishes the connection between the projected score function and the Fisher divergence.
\begin{lemma}[\citealp{feng2025optimal}]
%\label{lem:p0-star}
In the setting of Lemma~\ref{lem:psi0-star}, there is a unique continuous log-concave density~$f_0^*$ on $\mathbb{R}$ such that $\log f_0^*$ has right derivative $\psi_0^*$ on $\mathcal{S}_0$.  Furthermore, if $f_0$ is absolutely continuous, then~$f_0^*$ minimises $I(f_0,f)$ over $f \in \mathcal{F}_1$, and if $f_0 \in \mathcal{F}_1$, then $f_0^* = f_0$.
\end{lemma} 
Writing $f_0^{\mathrm{ML}} := \argmax_{f \in \mathcal{F}_1} \int_{-\infty}^\infty f_0 \log f$ for the maximum likelihood log-concave projection, the  $M$-estimators 
%\begin{equation}
%\label{eq:betahat-ML-fisher}
\[
\hat{\beta}_{\psi_0^{\mathrm{ML}}} \in \argmax_{\beta \in \mathbb{R}^d} \sum_{i=1}^n \log f_0^{\mathrm{ML}}(Y_i - X_i^\top\beta) \quad\text{and}\quad \hat{\beta}_{\psi_0^*} \in \argmax_{\beta \in \mathbb{R}^d} \sum_{i=1}^n \log f_0^*(Y_i - X_i^\top\beta)
\]
%\end{equation}
are generally different. In fact, the following result shows that there exist error distributions for which the asymptotic covariance of $\hat{\beta}_{\psi_0^{\mathrm{ML}}}$ is arbitrarily large compared with that of the optimal convex $M$-estimator $\hat{\beta}_{\psi_0^*}$, even when the latter is close to being asymptotically efficient.
\begin{proposition}[\citealp{feng2025optimal}]
%\label{prop:Vp0-MLE}
For every $\epsilon \in (0,1)$, there exists a distribution $P_0$ with a finite mean and an absolutely continuous density $f_0$ such that $i(f_0) < \infty$, and the log-concave maximum likelihood projection~$f_0^{\mathrm{ML}}$ has corresponding score function $\psi_0^{\mathrm{ML}}$ satisfying
%\begin{equation}
%\label{eq:Vp0-MLE}
\[
\frac{V_{f_0}(\psi_0^*)}{V_{f_0}(\psi_0^{\mathrm{ML}})} \leq \epsilon \quad\text{and}\quad \mathrm{ARE}^*(f_0) \geq 1 - \epsilon.
\]
%\end{equation}
\end{proposition}
The wider moral for shape-constrained estimation is that the notion of projection onto a shape-constrained class should be tailored to the task at hand.

Returning to our original linear regression problem, a natural estimation strategy on the population level is to alternate between the following two steps:
\begin{enumerate}[(i)]
\item For a fixed $\beta$, minimise the (convex) score matching objective $D_{q_\beta}(\psi)$ based on the density $q_\beta$ of $Y_1 - X_1^\top\beta$.
\item For a fixed decreasing and right-continuous $\psi$, minimise the convex function $\beta \mapsto \mathbb{E}\ell(Y_1 - X_1^\top \beta)$, where $\ell$ is a negative antiderivative of $\psi$.
\end{enumerate}
\citet{feng2025optimal} establish that, in the case where $f_0$ is symmetric and satisfies mild regularity conditions, an appropriate sample version of this algorithm yields an estimator $\hat{\beta}_n$ of $\beta_0$ with
\[
\sqrt{n}(\hat{\beta}_n - \beta_0) \stackrel{d}{\rightarrow} N_d\biggl(0, \frac{\{\mathbb{E}(X_1 X_1^\top)\}^{-1}}{i^*(f_0)}\biggr)
\]
as $n \rightarrow \infty$.  A similar result holds without the symmetry assumption on $f_0$, but where an explicit intercept term is present in the linear model.  Thus, $\hat{\beta}_n$ is $\sqrt{n}$-consistent and has the same limiting Gaussian distribution as the `oracle' convex $M$-estimator $\hat{\beta}_{\psi_0^*} := \argmin_{\beta \in \mathbb{R}^d} \sum_{i=1}^n \ell_0^*(Y_i - X_i^\top\beta)$, where $\ell_0^*$ denotes an optimal convex loss function with right derivative $\psi_0^*$.  In this sense, it is \emph{antitonically efficient}.

\section{Other modern applications of shape constraints}
\label{Sec:OtherExamples}

\subsection{Isotonic subgroup selection}

In regression settings, subgroup selection refers to the challenge of identifying a subset of the covariate domain on which the regression function satisfies a particular property of interest.  This is a post-selection inference problem, since the region is to be selected after seeing the
data, and yet we still wish to claim that with high probability, the regression function satisfies this property on the selected set. Important applications can be found in precision medicine, for instance, where the chances of a desirable health outcome may be highly heterogeneous
across a population, and hence the risk for a particular individual may be masked in a study representing the entire population.

A natural strategy for identifying such group-specific effects is to divide a study into two stages, where the first stage is used to identify a potentially interesting subset of the covariate domain, and the second attempts to verify that it does indeed have the desired property \citep{stallard2014adaptive}.  However, such a two-stage process may often be both time-consuming and potentially expensive due to the inefficient use of the data, and moreover the binary second-stage verification may fail. In such circumstances, we are unable to identify a further subset of the original selected set on which the property does hold.

In many applications, heterogeneity across populations may be characterised by monotonicity of a regression function in individual covariates. For instance, for individuals with hypertrophic cardiomyopathy, risk factors for sudden cardiac death (SCD) include family history of SCD, maximal heart wall thickness and left atrial
diameter \citep{omahony2014novel}. It is frequently of interest to identify a subset of the population deemed to be at low or high risk, for instance to determine an appropriate course of treatment. This amounts to identifying an appropriate superlevel set of the regression
function.

\citet{muller2025isotonic} introduce a framework that allows the identification of the $\tau$-superlevel set of an isotonic regression function, for some pre-determined level $\tau$. A key component of their formulation of the problem is to recognise that often there is an asymmetry to the
two errors of including points that do not belong to the superlevel set, and failing to include points that do. For instance, in the case of hypertrophic cardiomyopathy, a false conclusion that an individual is at low risk of sudden cardiac death within five years, and hence does not require an implantable cardioverter defibrillator \citep{omahony2014novel}, is more serious than the opposite form of error, which obliges a patient to undergo surgery and deal with the inconveniences of the implanted device.

Suppose that we are given $n$ independent copies of a covariate-response pair $(X,Y)$ having a distribution on $\mathbb{R}^d \times \mathbb{R}$ with coordinate-wise increasing regression function~$\eta$ given by $\eta(x) := \mathbb{E}(Y | X=x)$ for $x\in \mathbb{R}^d$.  Given a threshold $\tau \in \mathbb{R}$, and with  $\mathcal{X}_\tau(\eta) := \{x \in \mathbb{R}^d:\eta(x) \geq \tau\}$ denoting the $\tau$-superlevel set of $\eta$, the goal is to output an estimate $\hat{A}$ of $\mathcal{X}_\tau(\eta)$ with the first priority that it guards against the more serious of the two errors mentioned above.  Without loss of generality, this more serious error may be taken to be that of including points in $\hat{A}$ that do not belong to $\mathcal{X}_\tau(\eta)$, and we therefore require Type~I error control in the sense that $\hat{A} \subseteq \mathcal{X}_\tau(\eta)$ with probability at least $1-\alpha$, for some pre-specified $\alpha \in (0,1)$.  Subject to this constraint, we seek to maximise $\mu(\hat{A})$, where $\mu$ denotes the marginal distribution of~$X$.

The method of \citet{muller2025isotonic}, as implemented in the \texttt{R} package \texttt{ISS} \citep{Mueller2023ISS}, seeks to compute at each observation a $p$-value for the null hypothesis that the regression function is below $\tau$ based on an anytime-valid martingale procedure \citep{howard2021uniform}.  The monotonicity of the regression function implies logical relationships between these hypotheses, and \citet{muller2025isotonic} introduce a tailored multiple testing procedure with familywise error rate control.  The final output set $\hat{A}^{\mathrm{ISS}}$ is the upper hull of the observations corresponding to the rejected hypotheses.  An illustration in a bivariate example is given in Figure~\ref{fig:simple_illustration}.

\begin{figure}[htbp]
    \center
    \includegraphics[trim={2cm 2cm 2cm 5cm}, clip, width=\textwidth]{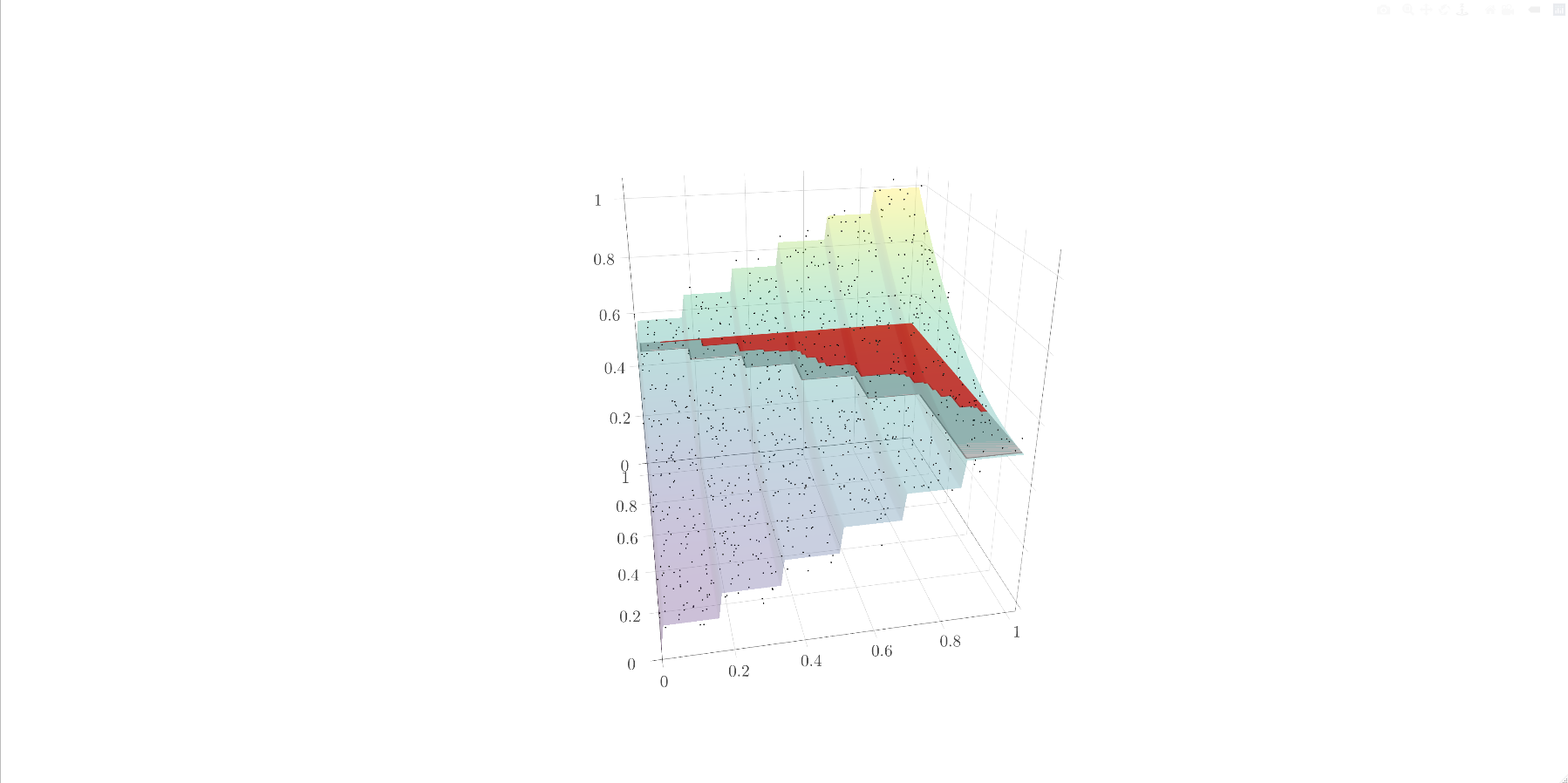}
    \vspace{-0.4cm}
    \caption{A visualisation with $d=2$ and $n = 1000$. The unknown regression function is depicted by the multi-coloured surface. The grey surface gives the $0.5$-super-level set, of which the red area is selected by $\hat{A}^\mathrm{ISS}$. }
    \label{fig:simple_illustration}
\end{figure}

\citet{muller2025isotonic} verify that $\hat{A}^{\mathrm{ISS}}$ does indeed control Type~I error in the sense outlined above.  Moreover, they provide a bound on $E\bigl\{\mu\bigl(\mathcal{X}_\tau(\eta) \setminus \hat{A}^{\mathrm{ISS}}\bigr)\bigr\}$, which in combination with a corresponding minimax lower bound reveals that $\hat{A}^{\mathrm{ISS}}$ minimises this expected regret up to poly-logarithmic factors, among all procedures that control the Type~I error.  The method, which is tuning-free, therefore offers a practical alternative to approaches that exploit smoothness of the regression function \citep[e.g.][]{reeve2023optimal}.

\subsection{Testing conditional independence}

Testing conditional independence underpins the problems of variable selection, graphical modelling and causal inference.  To formalise the setting, consider the null hypothesis
\[
H_0^{\mathrm{CI}}: X \ci Y \mid Z,
\]
where $X$ and $Y$ are variables of interest (such as a treatment $X$ and an outcome $Y$), while~$Z$ represents a (potentially high-dimensional) confounder.  Our available data consist of independent copies $(X_1,Y_1,Z_1),\ldots,(X_n,Y_n,Z_n)$ of $(X,Y,Z) \sim P$, for some unknown distribution $P$ on $(\mathcal{X},\mathcal{Y},\mathcal{Z})$.   The following remarkable result, however, illustrates that the problem of conditional independence testing is fundamentally~hard. %in a strong sense.

\begin{theorem}[\citealp{shah2020hardness}]
\label{Thm:Hardness}
    Let $\mathcal{P}_{\mathrm{AC}}$ denote the set of distributions on $\mathbb{R}^d$ that are absolutely continuous distributions with respect to Lebesgue measure.  For any $\alpha\in (0,1)$, any test of the null $H_0^{\mathrm{CI}} \cap \mathcal{P}_{\mathrm{AC}}$ with Type~I error level $\alpha$ has power no greater than $\alpha$ at every alternative distribution in $\mathcal{P}_{\mathrm{AC}} \setminus H_0^{\mathrm{CI}}$.
\end{theorem}

The lesson from Theorem~\ref{Thm:Hardness} is that some further restriction of the null hypothesis is necessary for a non-trivial test.  \citet{hore2025testing} proceed as follows:

\begin{customasm}{1}
\label{Ass:SI}
    Let $\mathcal{X} \subseteq \mathbb{R}$ and let $\preceq$ be a partial order on $\mathcal{Z}$.  Assume that $X$ is stochastically increasing in~$Z$, meaning that if $z\preceq z'$ then $\mathbb{P}(X \geq x \, | \, Z=z) \leq \mathbb{P}(X \geq x \, | \, Z=z')$ for all $x$.
\end{customasm}

This assumption is motivated by applications, particularly in biomedicine, where for instance factors such as smoking intensity may be associated with increased risk of certain diseases or conditions.  The idea for a test of the isotonic conditional independence null $H_0^{\mathrm{ICI}}$, i.e.~distributions satisfying conditional independence and Assumption~\ref{Ass:SI}, is to consider carefully-chosen pairs of data points $(X_i,Y_i,Z_i)$ and $(X_j,Y_j,Z_j)$ with $Z_i \preceq Z_j$.  Under $H_0^{\mathrm{ICI}}$, one expects $X_i \leq X_j$ more often than not, and a substantial violation of this provides evidence of the influence of $Y$, i.e.~evidence against $H_0^{\mathrm{ICI}}$.  To calibrate the test appropriately under the null, \citet{hore2025testing} employ a particular type of permutation test, where permutations are restricted within matched pairs.  The resulting \texttt{PairSwap-ICI} procedure guarantees finite-sample Type I error control over $H_0^{\mathrm{ICI}}$, and the power properties are characterised under a broad family of regression models for $X$ conditional on $(Y,Z)$.

\section{Outlook and open problems}

Looking to the future, we see great further potential for shape constraints to be incorporated into other common statistical tasks.  Given the scale and complexities of modern data sets now routinely collected, the flexibility of nonparametric approaches is extremely valuable.  Shape constraints often offer a viable and sometimes more appealing alternative to methods that rely on the smoothness of an unknown function, and moreover we may be able to eschew a delicate choice of tuning parameters.  The new approach to linear regression outlined in Section~\ref{Sec:LinearRegression} offers a glimpse of the aptitude of shape-constrained ideas in semiparametric problems, and we anticipate many further developments in related directions.

We conclude by mentioning four open problems related to log-concave density estimation (Section~\ref{Sec:LogConcaveDensityEstimation}):
\begin{enumerate}
\item Regarding computation of the log-concave MLE $\hat{f}_n$, is it possible to exploit a warm start if a new data point is added, or one is deleted?  The convex hull of the data can be triangulated into simplices on which $\log \hat{f}_n$ is affine, but at this time it is unknown how this structure is modified under perturbations of the data.
\item Suppose that $d \geq 2$, and that our data may be observed with missingness in some coordinates.  How can we best exploit the data with partial observations?  In an extreme version of this problem, we might assume that the marginal log-concave densities were known.
\item What can we say about the theoretical properties of the smoothed log-concave MLE?
\item What can we say about the boundary behaviour of the log-concave MLE in the multivariate case?  Recent work of \citet{ryter2024tails} provides some key answers in the univariate case.
\end{enumerate}

\textbf{Acknowledgements}: This research was supported by European Research Council Advanced Grant 101019498.

\bibliographystyle{apalike}
\bibliography{bib}

\end{document}